\documentclass[11pt,a4paper,reqno]{amsart}
\usepackage[T1]{fontenc}
\usepackage[utf8]{inputenc}
\usepackage{lmodern,microtype}
\usepackage[margin=25mm]{geometry}
\usepackage{amssymb,mathtools,mathrsfs}
\usepackage{booktabs}
\usepackage{enumitem}
\usepackage{xcolor}
\definecolor{inkblue}{HTML}{243D50}
\usepackage[colorlinks=true,linkcolor=inkblue,citecolor=inkblue,
  urlcolor=inkblue,pagebackref=true]{hyperref}
\usepackage{bookmark}
\hypersetup{
  pdftitle={Harmonic Norms on Quiver Representations over Non-Archimedean Fields},
  pdfauthor={Oren Ben-Bassat},
  pdfsubject={Harmonic norms, complete non-Archimedean fields, and quiver representations},
  pdfkeywords={quiver representations, harmonic norms, non-Archimedean fields,
    polystability, semistable reduction, split norms, Bruhat-Tits buildings,
    convex potentials}
}
\renewcommand*{\backrefalt}[4]{%
  \ifcase #1\relax
  \or\space\textit{Cited on page~#2.}%
  \else\space\textit{Cited on pages~#2.}%
  \fi}
\setlist{itemsep=3pt,topsep=5pt}
\numberwithin{equation}{section}
\newtheorem{theorem}{Theorem}[section]
\newtheorem{proposition}[theorem]{Proposition}
\newtheorem{lemma}[theorem]{Lemma}
\newtheorem{corollary}[theorem]{Corollary}
\theoremstyle{definition}
\newtheorem{definition}[theorem]{Definition}
\theoremstyle{remark}
\newtheorem{remark}[theorem]{Remark}
\newcommand{\R}{\mathbb R}
\newcommand{\Ndiag}{\mathcal N^{\mathrm{diag}}}
\newcommand{\Nall}{\mathcal N}
\newcommand{\GF}{\mathsf{GF}}
\newcommand{\gr}{\operatorname{gr}}
\newcommand{\rk}{\operatorname{rk}}
\DeclareRobustCommand{\namedref}[2]{\hyperref[#2]{#1~\ref*{#2}}}
\makeatletter
\DeclareRobustCommand{\fulleqref}[1]{%
  \hyperref[#1]{\textup{\tagform@{\ref*{#1}}}}}
\DeclareRobustCommand{\namedeqref}[2]{%
  \hyperref[#2]{#1~\textup{\tagform@{\ref*{#2}}}}}
\makeatother
\newenvironment{sourcequote}[1]{\begin{samepage}\par\smallskip\noindent\textbf{#1}\begin{quote}}{\end{quote}\smallskip\end{samepage}}

\title[Harmonic norms over non-Archimedean fields]{Harmonic Norms on Quiver Representations over Non-Archimedean Fields}
\author{Oren Ben-Bassat}
\address{University of Haifa}
\email{ben-bassat@math.haifa.ac.il}
\date{\today}

\begin{document}

\begin{abstract}
Let $K$ be any complete non-Archimedean field, not necessarily spherically
complete. We prove that every polystable representation of a finite quiver
over $K$ admits a split harmonic norm in the sense of
Haiden--Katzarkov--Kontsevich--Pandit (HKKP), answering a question from
their preprint \emph{Towards Categorical K\"ahler Geometry}.
Moret-Bailly's closed-image theorem gives a uniform separation from
representations with destabilizing subobjects. A finite estimate on norm
weights then yields a linear coercivity bound for stable representations
and hence a minimum of the extended HKKP potential on the complete CAT(0)
space of all norms. A direct first-variation computation identifies split
minimizers with HKKP harmonic norms. Strict supporting hinges reduce split
attainment to semistable reduction for quivers with arbitrary real radius
labels. Small changes of the lifted splitting make each reversed extension
nonsplit; a finite slope--rank descent proves termination. Consequently,
every minimum in the completed norm space is attained by a split norm
with the same arrow singular profiles, without restrictions on the value
group or the quiver. Nonsplit minimizers illustrate split attainment in
one and two radius modifications, and a two-parameter Berkovich analytic family separates
polystability, semistability, and harmonicity of a chosen split norm.
\end{abstract}
\maketitle

\clearpage
\tableofcontents
\bigskip

\section{Introduction}\label{sec:introduction}

Let $K$ be a field with a \emph{non-Archimedean absolute value},
that is, a map $|\cdot|:K\to\R_{\geq0}$ satisfying
\[
 |x|=0\Longleftrightarrow x=0,\qquad |xy|=|x||y|,\qquad
 |x+y|\leq\max\{|x|,|y|\}.
\]
We assume that $K$ is complete for $|\cdot|$.  The trivial absolute value is
allowed, and the value group need not be discrete.

A \emph{closed ball} in $K$ is a subset of the form
$B(a,r):=\{x\in K\mid |x-a|\leq r\}$, where
$a\in K$ and $r\in\R_{\geq0}$.  The field $K$ is
\emph{spherically complete} if every nonempty family of closed balls that is
totally ordered by inclusion has nonempty intersection.  We do not assume
that $K$ is spherically complete.

If $|\cdot|$ is trivial, every closed ball is either a singleton or all of
$K$.  Thus $K$ is spherically complete, every finite-dimensional norm is
split, and HKKP's theorem already applies.  The new argument below concerns
nontrivially valued fields; we retain the trivial case in the statements for
uniformity.

A finite quiver $Q$ has vertex set $Q_0$, arrow set $Q_1$, and source and
target maps \[s,t:Q_1\to Q_0.\]  A representation $E$ assigns a
finite-dimensional $K$-space $E_i$ to each vertex and a map
$\rho(a):E_{s(a)}\to E_{t(a)}$ to each arrow.  For
$\theta=(\theta_i)_{i\in Q_0}\in\R^{Q_0}$ put
\[
 \theta(E)=\sum_{i\in Q_0}\theta_i\dim E_i.
\]
We call $\theta(E)$ the \emph{$\theta$-degree} of $E$, or simply its degree
when $\theta$ is fixed.
Our convention is that a nonzero representation $E$ is $\theta$-stable if
\begin{equation}\label{eq:stability}
 \theta(E)=0,\qquad \theta(U)>0
 \quad\text{for all subrepresentations }\{0\}\ne U\subsetneq E.
\end{equation}
We call $E$ $\theta$-semistable if $\theta(E)=0$ and the same condition
holds with weak inequalities.  It is $\theta$-polystable if it is a finite direct sum of
$\theta$-stable representations; the zero representation is allowed as the
empty direct sum.

A classical precedent for existence questions of this kind is the
Donaldson--Uhlenbeck--Yau theorem, a form of the Kobayashi--Hitchin
correspondence: a holomorphic vector bundle on a compact K\"ahler manifold is
polystable precisely when it admits a Hermitian--Yang--Mills metric
\cite{Donaldson,UY,Lubke}.  Its finite-dimensional counterpart is the Kempf--Ness
theorem: for a complex reductive group $G$ acting linearly on a finite-dimensional
complex vector space $V$, after choosing a Hermitian norm invariant under a
maximal compact subgroup of $G$, an orbit $G\cdot v\subseteq V$ is closed precisely
when the norm attains a minimum on $G\cdot v$ \cite{KempfNess}.  King applied this principle to
quiver representations over $\mathbb C$. For integral stability weights,
$\theta$-polystability is equivalent to the existence of Hermitian metrics on
the vertex spaces satisfying the quiver moment-map equation
\cite[Proposition~6.5]{King}.  \'Alvarez-C\'onsul and Garc{\'\i}a-Prada later
extended the stability--metric correspondence to twisted quiver bundles via
the quiver vortex equations \cite[Theorem~3.1]{AlvarezGarciaPrada}.

On the non-Archimedean side, Goldman and Iwahori introduced an intrinsic
metric on the space of $p$-adic norms, viewed as an analogue of the symmetric
space of positive-definite forms \cite{GoldmanIwahori}.  Bruhat and Tits
placed this example in their general theory of buildings for reductive groups
\cite{BruhatTits}.  For a finite-dimensional $K$-space $V$, the
diagonalizable norms form the extended Euclidean Bruhat--Tits building of
$\operatorname{GL}(V)$ \cite[Example~3.2]{BE}, while their homothety classes
form its reduced building, of affine type $\widetilde A_{\dim V-1}$ when
$\dim V\geq2$
\cite[Remark~1.23 and Theorem~1.25]{RTW}.  Our building terminology and the
precise norm models are specified in \namedref{Subsection}{subsec:normspaces}.
Maculan later developed a non-Archimedean Kempf--Ness theory in Berkovich GIT
\cite[Chapter~3]{Maculan}; the norm-minimization problem studied here is a
quiver analogue of that variational theory.  In the quiver setting, Haiden, Katzarkov,
Kontsevich, and Pandit define harmonic norms through graded reduction and a
convex potential on the space of norms.  Spherical completeness makes every
finite-dimensional norm split \cite[Proposition~6.9]{HKKP}; vertexwise this
gives $\Ndiag(E)=\Nall(E)$, so the product extended norm building is then a
complete CAT(0) space \cite[Corollary~3.3]{BE}.  Their minimization proof uses
this hypothesis.  Accordingly, $K$ is assumed to be spherically complete
throughout HKKP's Section~6.3 \cite{HKKP}; they also use the standing
convention that quivers are finite.  Making the field hypothesis
explicit, their result reads as follows.
\begin{sourcequote}{Adapted from Haiden--Katzarkov--Kontsevich--Pandit,
\cite[Theorem~6.23, p.~76]{HKKP}.}
\textbf{Theorem 6.23.} Let $K$ be a spherically complete non-Archimedean
field, $Q$ a quiver, $\theta=(\theta_i)_{i\in Q_0}\in\R^{Q_0}$,
$f=(f_a)_{a\in Q_1}$ a family of weak ramp functions, and
$E=((E_i)_i,\rho)$ a representation of $Q$ over $K$. Then the following
implications hold:
\[
 \begin{aligned}
 E\text{ is $\theta$-polystable}
 &\Longrightarrow E\text{ admits a harmonic norm for }(\theta,f)\\
 &\Longrightarrow E\text{ is $\theta$-semistable}.
 \end{aligned}
\]
\end{sourcequote}
They ask immediately afterward:
\begin{sourcequote}{Haiden--Katzarkov--Kontsevich--Pandit,
\cite[Question~6.24, p.~76]{HKKP}.}
\phantomsection\label{quote:hkkpquestion624}
\textbf{Question 6.24.} Can the spherical completeness assumption on $K$ be
omitted in Theorem~6.23?
\end{sourcequote}
\namedref{Theorem}{thm:mainintro} answers this question affirmatively.

\begin{theorem}\label{thm:mainintro}
Fix a finite quiver $Q$, a complete non-Archimedean field $K$, a real
stability parameter $\theta\in\R^{Q_0}$, a family
$f=(f_a)_{a\in Q_1}$ of weak ramps, and a finite-dimensional representation
$E$ of $Q$ over $K$.
Choose nonzero determinant vectors in the spaces $\det E_i$, and let $S_E$
be the resulting extended HKKP potential.  Then
\[
 \begin{aligned}
 E\text{ is $\theta$-polystable}
 &\Longrightarrow E\text{ admits a split HKKP harmonic norm for }(\theta,f)\\
 &\Longrightarrow E\text{ is $\theta$-semistable}.
 \end{aligned}
\]
Moreover, if $E$ is $\theta$-polystable, $S_E$ admits a global minimizer
on $\Nall(E)$, the complete space of all norms on $E$.
Furthermore, if $\theta(E)=0$ and $h\in\Nall(E)$ globally minimizes
$S_E$, there is a split global minimizer $k\in\Ndiag(E)$ with
the same arrow singular profiles (the ordered lists of logarithmic singular
values of the arrow maps, as defined in~\fulleqref{eq:extendedsingularvalues})
and the same potential value as $h$.
\end{theorem}

The space $\Nall(E)$ is the metric completion of the
product extended norm building $\Ndiag(E)$.  These spaces are specified in
\namedref{Definition}{def:euclideanbuilding},
\namedref{Definition}{def:normsapartments}, and
\namedref{Definition}{def:metrizations}. The HKKP potential $S_E$ is
defined in \fulleqref{eq:HKKPpotential} and extended in
\namedref{Lemma}{lem:allnormprofile}. Existence of a minimizer in the
complete space of norms is proved in
\namedref{Theorem}{thm:universalvariational}.  The profile-preserving
split-attainment assertion and the two displayed implications are proved in
\namedref{Theorem}{thm:splitattainment}, using the split-minimizer criterion
\namedref{Corollary}{cor:splitcriterionallfields}.

The proof has two parts: existence of a minimizer of $S_E$ within the
completed space $\Nall(E)$, and attainment of the same minimum by a split
norm in $\Ndiag(E)$. These questions coincide over a spherically complete
field. For existence, consider the proper incidence morphisms that
parametrize representations with subrepresentations of prescribed dimension
vectors. Moret-Bailly's closed-image theorem \cite[Theorem~1.3]{MB}
separates a stable representation uniformly from the images corresponding
to nonpositive-degree subobjects. In coordinates orthonormal for a fixed
lattice norm, this gives a lower bound for a forbidden arrow coefficient
at every nonpositive-degree cut. Summing the gaps between the norm weights
then gives constants $c>0$ and $C_0\in\R$ such that
\[
 S_E(h)\geq c\,\bar d_\infty([h],[h^0])-C_0.
\]
Thus the sublevel sets modulo simultaneous homothety are bounded. The
direct method on the complete CAT(0) space gives a global minimizer
\cite[Lemma~2.2.19, p.~42]{Bacak}; direct sums treat the polystable case.
The split-minimizer criterion follows separately from a first-variation
calculation over $K$, using the graded singular filtrations.

For the split-attainment part, we seek $k\in\Ndiag(E)$ in the split
locus of \namedref{Definition}{def:metrizations} with $S_E(k)=S_E(h)$.
\namedref{Lemma}{lem:supportinghinges} defines $G_E$ so that $h$ minimizes it
and every global minimizer $k$ satisfies
\fulleqref{eq:hingeprofileequality} and $S_E(k)=S_E(h)$.  Thus it suffices to
find a split global minimizer of $G_E$. The max-envelope construction of \namedref{Definition}{def:maxenvelope}
and the determinant identity of
\namedref{Lemma}{lem:hingedeterminant}
convert this to minimizing an auxiliary determinant functional $J$. This functional is
an instance of \namedref{Definition}{def:labelleddetfunctional}, over the
$b$-bounded norms of
\namedref{Definition}{def:radiuslabelledreduction}.  In
\namedref{Lemma}{lem:radiuslangton}, each nonterminal modification
$Y\mapsto Y'$ strictly decreases, in lexicographic order, the pair consisting
of the maximal slope and the total rank of the maximal-slope subobject.
This subobject is characterized in
\namedref{Lemma}{lem:maximalslope}. The finitely many possible pairs force
termination at a split $b$-bounded norm $p$ with semistable graded reduction.
\namedref{Lemma}{lem:labelleddetminimum} then proves that $p$ minimizes $J$
on the full $b$-bounded locus, and
\fulleqref{eq:auxiliaryelimination} gives the required split minimizer of
$G_E$.  This is \namedref{Theorem}{thm:splitattainment}.

We give an example of a Berkovich analytic family of representations of $D_4$  over
the Novikov field in
\namedref{Subsection}{subsec:d4analyticfamily}, based on \cite[Example~6.22]{HKKP}. It
explicitly separates the locus of polystable fibers, the locus where the split norm is harmonic, and the locus of semistable fibers.
The semistable locus is also the locus where a harmonic norm exists; it
is strictly larger than the locus where the chosen norm is harmonic. In particular, a harmonic
norm can globally minimize the potential on an indecomposable, nonpolystable
representation, as HKKP already observed.

\namedref{Subsection}{subsec:nonsplitattainment} exhibits an example of a nonsplit minimizer over
the Novikov field and follows the supporting-hinge construction through one
radius modification to a split minimizer with the same singular profile.
Finally, \namedref{Subsection}{subsec:d4nonsplit} describes the entire
minimum locus of a $D_4$ representation in terms of norms on its
two-dimensional multiplicity space, and gives a split-attainment
calculation requiring two successive radius modifications.
\bigskip

\noindent\textbf{Acknowledgments.} This article was written through extensive conversations with OpenAI's ChatGPT-6.0 Sol, beginning with the author's original arguments, which contained substantial gaps. The ideas and development of the article represent roughly equal contributions of the author and the AI.
OpenAI's ChatGPT, Codex, and Prism were used to discuss and revise the
arguments and to proofread and typeset the manuscript.  The examples were chosen by the author.

\section{Norms, singular filtrations, and graded flags}\label{sec:prelim}

\subsection{Norm spaces and buildings}\label{subsec:normspaces}

Write
\[
 \begin{gathered}
 K^\circ=\{a\in K\mid |a|\leq1\},\qquad
 K^{\circ\circ}=\{a\in K\mid |a|<1\},\\
 \widetilde K=K^\circ/K^{\circ\circ},\qquad
 \Gamma_K=-\ln|K^\times|\subseteq\R.
 \end{gathered}
\]
Here $K^\circ$ is the valuation ring, $K^{\circ\circ}$ its maximal ideal,
$\widetilde K$ its residue field, and $\Gamma_K$ the additive value group.

\begin{definition}\label{def:euclideanbuilding}
Let $(A,W)$ be model-apartment data as in
\cite[Definition~1.9]{RTW}: $A$ is a Euclidean affine space with underlying
vector space $A_{\mathrm{vect}}$, the vectorial part
$W_0\leq O(A_{\mathrm{vect}})$ of $W$ is a finite reflection group, and
there exist a special point $x\in A$ and a possibly nondiscrete translation
subgroup $T$ such that $W=W_xT$, where
$W_x=\{w\in W\mid w(x)=x\}$ is the stabilizer of $x$.
A \emph{Euclidean building} $X$ modeled on $(A,W)$ is a space
with an apartment atlas satisfying axioms \textup{(EB1)}--\textup{(EB5)} of
that definition.
In particular, charts are stable under precomposition by $W$, any two points
lie in a common apartment, chart changes on apartment intersections are
restrictions of elements of $W$, and any two Weyl sectors contain subsectors
in a common apartment.  For every apartment $A'$ and every $y\in A'$, axiom
\textup{(EB5)} also supplies a distance-nonincreasing retraction
$r_{y,A'}:X\to A'$ such that
\[
 r_{y,A'}|_{A'}=\operatorname{id}_{A'},\qquad
 r_{y,A'}^{-1}(y)=\{y\}.
\]
The induced geodesic metric makes the chart images, called
\emph{apartments}, isometric copies of $A$. Thus a building is
assembled by gluing copies of its model apartment, which provide its local
coordinates.  In this paper, an \emph{extended Euclidean building} means a metric product
$X_{\mathrm{ess}}\times\R^m$ on whose Euclidean factor $W_0$ acts trivially;
$X_{\mathrm{ess}}$ is its \emph{essential}, or \emph{reduced}, building;
compare \cite[Remark~1.23]{RTW}.
\end{definition}

\begin{definition}\label{def:normsapartments}
Let $V$ be a finite-dimensional $K$-vector space. A
\emph{non-Archimedean norm} on $V$ is a function
$h:V\to\R_{\geq0}$ such that $h(v)=0$ if and only if $v=0$ and, for
all $a\in K$ and $v,w\in V$,
\[
 h(av)=|a|h(v),\qquad h(v+w)\leq\max\{h(v),h(w)\}.
\]
A basis
$\mathbf e=(e_1,\ldots,e_n)$ is \emph{$h$-orthogonal} if
\[
 h\!\left(\sum_{j=1}^n a_je_j\right)
   =\max_{1\leq j\leq n}|a_j|h(e_j)
 \qquad\text{for all }a_1,\ldots,a_n\in K.
\]
An $h$-orthogonal basis $\mathbf e$ is \emph{$h$-orthonormal} if $h(e_j)=1$ for every
$j$.  The norm $h$ is \emph{split}, or \emph{diagonalizable}, if it
admits an orthogonal basis.  We write $\Nall(V)$ for all norms on $V$ and
$\Ndiag(V)\subseteq\Nall(V)$ for the split norms.  For a fixed basis
$\mathbf e$, its \emph{apartment} is
\[
 \mathcal A_{\mathbf e}
 :=\{h_{\mathbf e,\mathbf w}\mid \mathbf w\in\R^n\}
 \subseteq\Ndiag(V),
\]
where
\[
 h_{\mathbf e,\mathbf w}\!\left(\sum_{j=1}^n a_je_j\right)
 :=\max_{1\leq j\leq n}e^{w_j}|a_j|.
\]
\end{definition}

Thus $w_j=\ln h_{\mathbf e,\mathbf w}(e_j)$ and
$\Ndiag(V)=\bigcup_{\mathbf e}\mathcal A_{\mathbf e}$.  The terminology agrees
with \namedref{Definition}{def:euclideanbuilding}: after equipping
$\Ndiag(V)$ with the metric $d_2$ below, each $\mathcal A_{\mathbf e}$ is an
isometric copy of the model apartment $\R^n$, and these copies cover the
building.

  The Goldman--Iwahori distance on
$\Nall(V)$ is
\[
d_\infty(h,k)=\sup_{v\ne0}|\ln k(v)-\ln h(v)|.
\]
By finite-dimensional norm equivalence \cite[Proposition~1.6]{BE}, that supremum is finite when $V\ne\{0\}$.
The space $\Nall(V)$ is complete for $d_\infty$, and $\Ndiag(V)$ is
dense in it \cite[Proposition~1.8 and Theorem~1.19(ii)]{BE}.

Any two split norms lie in a common apartment
\cite[Proposition~1.14]{BE}.  If
$h=h_{\mathbf e,\mathbf u}$ and $k=h_{\mathbf e,\mathbf w}$ lie in
$\mathcal A_{\mathbf e}$, put
\[
 \lambda_j=w_j-u_j=\ln\frac{k(e_j)}{h(e_j)},
 \qquad 1\leq j\leq n.
\]
Then $d_\infty(h,k)=\max_j|\lambda_j|$.  The common-apartment expression
\[
 d_2(h,k)=\left(\sum_{j=1}^n|\lambda_j|^2\right)^{1/2}
\]
is apartment-independent: up to order, the $\lambda_j$ are the intrinsically
defined relative spectrum of $(h,k)$
\cite[Definition~2.22 and Proposition~2.24]{BE}.  By
\cite[Theorem~3.1]{BE}, it defines a metric $d_2$ on $\Ndiag(V)$.
Thus, for $n=\dim V$,
\begin{equation}\label{eq:d2dinftycomparison}
 d_\infty(h,k)\leq d_2(h,k)\leq\sqrt n\,d_\infty(h,k).
\end{equation}
\begin{remark}

When $V=\{0\}$, the empty basis is orthogonal by convention,
$\Nall(V)=\Ndiag(V)$ consists of the unique norm $h_\varnothing(0)=0$, and
the preceding coordinate formula is interpreted as this unique norm. We set
$d_\infty=d_2=0$, identify $\det V$ with $K$, choose determinant vector $1$
of norm one, and interpret empty suprema, metric sums, and maxima as zero.
\end{remark}
By \cite[Example~3.2]{BE}, the charts
$\mathbf w\mapsto h_{\mathbf e,\mathbf w}$ make $(\Ndiag(V),d_2)$ the
\emph{extended norm building} of $\operatorname{GL}(V)$.  The coordinate
description above shows that its model apartment is $\R^n$ and that its
linear Weyl group $S_n$ permutes the coordinates.  The affine apartment transformations
induced by monomial changes of basis form
$\Gamma_K^n\rtimes S_n$, the extended affine Weyl group.   For $c>0$ and
$h\in\Nall(V)$, write $ch$ for the norm $(ch)(v)=c\,h(v)$.  Two norms
$h,k$ are \emph{homothetic} if $k=ch$ for some $c>0$, and we denote the
homothety class of $h$ by $[h]$.  Scaling by $c$ translates every apartment
by $(\ln c)(1,\ldots,1)$.  Hence, when $n>0$, the homothety quotient
$\Ndiag(V)/\R_{>0}$ is the reduced building of $\operatorname{GL}(V)$,
equivalently the Bruhat--Tits building of $\operatorname{PGL}(V)$ or
$\operatorname{SL}(V)$.  It is a point for $n=1$ and has affine type
$\widetilde A_{n-1}$ for $n\geq2$
\cite[Remark~1.23 and Theorem~1.25]{RTW}.  Its apartments are
$\R^n/\R(1,\ldots,1)$, canonically isometric by orthogonal projection to
$\{\mathbf x\in\R^n\mid\sum_jx_j=0\}$.  The induced extended affine Weyl
group is $(\Gamma_K^n/\Gamma_K(1,\ldots,1))\rtimes S_n$.  Its affine Weyl
subgroup generated by affine-root (wall) reflections has translation part equal to the
image of
\[
 \{(\gamma_1,\ldots,\gamma_n)\in\Gamma_K^n
   \mid \textstyle\sum_j\gamma_j=0\}.
\]
The group $\operatorname{GL}(V)$ acts by
$(g\cdot h)(v)=h(g^{-1}v)$ and carries $\mathcal A_{\mathbf e}$ to
$\mathcal A_{g\mathbf e}$; its action on homothety classes factors through
$\operatorname{PGL}(V)$.  Thus ``building'' without the adjective
``reduced'' refers below to the unquotiented, extended norm building.

For $p\in\{2,\infty\}$, the quotient metric on homothety classes is
\[
 \bar d_p([h],[k])=\inf_{c>0}d_p(h,ck).
\]
If $n>0$ and $h,k$ lie in a common apartment and their logarithmic coordinate
differences are $\lambda_1,\ldots,\lambda_n$, then
\[
 \bar d_\infty([h],[k])=
 \frac12\left(\max_j\lambda_j-\min_j\lambda_j\right).
\]

The building metric $d_2$ extends uniquely to a complete CAT(0) metric on
$\Nall(V)$ \cite[Corollary~3.3]{BE}. The CAT(0) condition means that
distances within a geodesic triangle are no greater than the corresponding
distances in its Euclidean comparison triangle; in particular, geodesics
are unique. Since $\Ndiag(V)$ is dense for
both metrics and~\fulleqref{eq:d2dinftycomparison} is uniform, the same
inequalities hold on the completion:
\[
 d_\infty(h,k)\leq d_2(h,k)\leq\sqrt n\,d_\infty(h,k),
 \qquad h,k\in\Nall(V).
\]
In general $\Nall(V)$ is not the norm building just described: a nonsplit
norm lies in none of its basis apartments.  Rather, it is the complete
CAT(0) metric completion of the extended building $\Ndiag(V)$.  If every
finite-dimensional norm over $K$
is split---in particular, if $K$ is spherically complete---these two spaces
coincide and the building itself is complete.

\begin{definition}\label{def:canonicalgeodesics}
For $h,k\in\Nall(V)$, the \emph{canonical geodesic} from $h$ to $k$ is
the unique constant-speed $d_2$-geodesic
$[h,k]:[0,1]\to\Nall(V)$.  If $h$ and $k$ are split, choose a common
apartment and write
$h=h_{\mathbf e,\mathbf u}$ and $k=h_{\mathbf e,\mathbf w}$.  Then
\[
 [h,k]_t=h_{\mathbf e,(1-t)\mathbf u+t\mathbf w},
 \qquad 0\leq t\leq1.
\]
This formula is independent of the chosen common apartment by uniqueness.
\end{definition}

The distance $d_\infty$ is used for approximation and quantitative
estimates.  When $\dim V>0$,
the determinant norm $h_{\det}$ is the quotient norm on $\det V$ induced by
the projective ultrametric tensor norm; an explicit formula is recalled in
\fulleqref{eq:exteriornorm}.

For a complete valued extension $L/K$, the scalar extension of
$h\in\Nall(V)$ is the map $h_L:V\otimes_KL\to\R_{\geq0}$ defined by
\[
 h_L(w)=\inf_{\substack{w=\sum_j a_jv_j\\a_j\in L,\ v_j\in V}}
          \max_j |a_j|h(v_j),
 \qquad w\in V\otimes_KL.
\]
It restricts to $h$, the map $h\mapsto h_L$ is an isometric embedding for
$d_\infty$, and every orthogonal basis of a split norm remains orthogonal
with the same weights
\cite[Definition~1.24 and Proposition~1.25(i), (ii), and (iv)]{BE}.

A linear map $g:(V,h)\to(W,k)$ is \emph{non-expansive} if
$k(gv)\leq h(v)$ for every $v\in V$, equivalently if its operator norm is at
most one.

\subsection{Subspace, quotient, and determinant norms}

For a finite-dimensional $K$-vector space $V$ and a $K$-linear subspace
$W\subseteq V$, restriction is the map
$\Nall(V)\to\Nall(W)$, $h\mapsto h|_W$. For a fixed nonzero
determinant vector $\omega_W\in\det W$, define
$\ell_W:\Nall(V)\to\R$ by
\[
 \ell_W(h)=\ln ((h|_W)_{\det}(\omega_W)).
\]
Changing $\omega_W$ changes every functional below in which $\ell_W$ occurs
by an additive constant.  In particular, it changes $S_E$, $G_E$, and the
determinant functional $J$ below only by additive constants; their minimizers
and their sublevel sets, after translating the levels, are unchanged.  For a
line with chosen vector $v$, this is simply $\ln h(v)$.
For $h\in\Nall(V)$, the induced quotient norm
$h/W\in\Nall(V/W)$ is defined by
\[
 (h/W)(v+W):=\inf_{w\in W}h(v+w).
\]
This is a genuine norm: finite-dimensional norm equivalence gives a split
norm $q$ and $C\geq1$ with $C^{-1}q\leq h\leq Cq$, and taking infima
gives $C^{-1}(q/W)\leq h/W\leq C(q/W)$. The split quotient norm $q/W$
is positive on nonzero cosets by \namedref{Lemma}{lem:adapted} below.

\begin{lemma}\label{lem:adapted}
A subspace of a finite-dimensional split normed space over $K$ is split
\cite[Lemma~1.13]{BE}. The corresponding quotient, with its quotient norm, is
also split. Every prescribed flag of subspaces admits an adapted
orthogonal basis. A basis is \emph{adapted} to a flag if every member of
the flag is spanned by a subset of that basis.
\end{lemma}
\begin{proof}
Fix an orthogonal basis $(b_1,\ldots,b_n)$ of $V$. If
$v=\sum_jc_jb_j\ne0$, choose $p$ such that
$h(v)=|c_p|h(b_p)$. The basis obtained by replacing $b_p$ with $v$ is
orthogonal. Indeed, equip the replacement vector $v$ with its actual weight
$h(v)=|c_p|h(b_p)$. Both the resulting change-of-coordinate map between the
weighted max-normed coordinate spaces and its inverse are non-expansive because
$|c_j|h(b_j)\leq |c_p|h(b_p)$ for every $j$. Hence
\[
 V=Kv\mathbin{\widehat\oplus}V',
 \qquad V'=\operatorname{span}\{b_j\mid j\ne p\},
\]
is an orthogonal direct sum, meaning
$h(av+v')=\max\{|a|h(v),h(v')\}$ for $a\in K$ and $v'\in V'$.

Discard repetitions from the prescribed flag and argue by induction on
$\dim V$. Choose $v$ in its smallest nonzero member. Every later member
$F$ contains $v$, and the displayed decomposition gives
\[
 F=Kv\mathbin{\widehat\oplus}(F\cap V').
\]
The intersections with $V'$ form a flag in the lower-dimensional split
space $V'$. Applying the induction hypothesis there and adjoining $v$
produces an orthogonal basis adapted to the original flag. The assertion
for one subspace is the special case of the flag $\{0\}\subseteq W\subseteq V$.
For this flag, the images in $V/W$ of the complementary adapted basis
vectors form an orthogonal basis for the quotient norm.
\end{proof}

\begin{definition}\label{def:metrizations}
For a representation $E=(E_i,\rho(a))$ of $Q$, a \emph{norm on $E$} is a family
$h=(h_i)_{i\in Q_0}$ with $h_i\in\Nall(E_i)$ for every vertex $i$.  It is
\emph{split} if every $h_i$ is split, and it is a \emph{lattice
norm} if every $h_i$ admits an orthonormal basis.  Two norms
$h=(h_i)_i$ and $k=(k_i)_i$ are \emph{simultaneously homothetic} if there
exists a single positive real number $c$ such that
$k_i(v)=c\,h_i(v)$ for every $i\in Q_0$ and every $v\in E_i$; again
we denote the simultaneous-homothety class of $h$ by $[h]$.  We write
\[
\Nall(E)=\prod_{i\in Q_0}\Nall(E_i),\qquad
\Ndiag(E)=\prod_{i\in Q_0}\Ndiag(E_i),
\]
for the spaces of all norms and split norms, respectively.
The subspace $\Ndiag(E)\subseteq\Nall(E)$ is called the \emph{split
locus}.  We write $\overline{\Ndiag}(E)$ and $\overline{\Nall}(E)$ for the
respective quotients by simultaneous homothety and call
$\overline{\Ndiag}(E)$ the \emph{quotient split locus}.
\end{definition}

With the Euclidean product metric
\[
 d_2(h,k)^2=\sum_{i\in Q_0}d_2(h_i,k_i)^2,
\]
$\Ndiag(E)=\prod_i\Ndiag(E_i)$ is the \emph{product extended norm
building}, namely the extended building of
$\mathbf G_E=\prod_{i\in I}\operatorname{GL}(E_i)$, where
$I=\{i\in Q_0\mid\dim E_i>0\}$.  The group $\mathbf G_E$ acts vertexwise, and the
apartments are the products $\prod_{i\in I}\mathcal A_{\mathbf e_i}$.
This building depends only on the vertex spaces of $E$, not on its arrow
maps.  Its canonical geodesics are taken vertexwise with the same parameter.
If $I=\varnothing$, all four norm spaces and quotients just defined are
points.  Suppose $I\ne\varnothing$.  The quotient
$\overline{\Ndiag}(E)$ removes only the one common scaling line.  Its
apartments are
\[
 \left(\prod_{i\in I}\R^{\dim E_i}\right)
 \big/\R\bigl((1,\ldots,1)_i\bigr)_{i\in I},
\]
and it is the extended building of
$\bigl(\prod_{i\in I}\operatorname{GL}(E_i)\bigr)/\mathbb G_m^{\mathrm{diag}}$.
In particular, it is not the product of the vertexwise reduced buildings:
it retains the $|I|-1$ relative vertex-scaling directions.

\begin{proposition}\label{prop:productcompletion}
The metric space $(\Nall(E),d_2)$ is complete and CAT(0), and it is the metric
completion of the product extended norm building $(\Ndiag(E),d_2)$.
\end{proposition}
\begin{proof}
Apply \cite[Corollary~3.3]{BE} vertexwise and take a finite Euclidean product.
\end{proof}

For $h=(h_i)_{i\in Q_0}$ and $k=(k_i)_{i\in Q_0}$ in $\Nall(E)$, set
\[
 d_\infty(h,k):=\max_{i\in I}d_\infty(h_i,k_i).
\]
The empty maximum is zero. For $p\in\{2,\infty\}$, the product quotient
metric is
\[
 \bar d_p([h],[k])=\inf_{c>0}d_p(h,ck).
\]
For $p=2$, the determinant-normalized slice in
\namedref{Lemma}{lem:quotientcompletion} meets every common scaling orbit
once and is isometric to the quotient for $\bar d_2$.  That lemma also
shows that $\overline{\Nall}(E)$ is the $\bar d_2$-metric completion of
$\overline{\Ndiag}(E)$; the uniform comparison
\fulleqref{eq:d2dinftycomparison} gives the same completion statement for
$\bar d_\infty$.

\subsection{Weak ramps and the harmonicity test}

In the following restatement of \cite[Definition~6.17]{HKKP}, a weak ramp
is a continuous function $f:[-\infty,\infty)\to\R$ whose restriction to
$\R$ is $C^1$ and has strictly increasing derivative satisfying
\[
 f'(\lambda)>0,\qquad \lim_{\lambda\to-\infty}f'(\lambda)=0,
 \qquad \lim_{\lambda\to\infty}f'(\lambda)=\infty.
\]
In particular $f':\R\to(0,\infty)$ has a continuous inverse
$(f')^{-1}:(0,\infty)\to\R$.

Fix real vertex weights $\theta=(\theta_i)_{i\in Q_0}$ and one weak ramp
$f_a$ for each arrow $a\in Q_1$.  These data determine the degree
homomorphism $\theta^\star$ below and hence the notion of harmonicity for a
normed representation of $Q$.  Thus harmonicity is always relative
to $\theta$ and $(f_a)_{a\in Q_1}$.

The following definitions of the associated graded spaces and singular
filtrations used in the harmonicity test restate
\cite[Section~6.2 and Definition~6.14]{HKKP}. Choose
once and for all a set $A\subseteq\R$ of representatives of
$\R/\Gamma_K$. For a split norm $h$ and $\alpha\in A$, put
\[
 V_\alpha^h=\{h\leq e^\alpha\}/\{h<e^\alpha\},\qquad
 \gr_hV=\bigoplus_{\alpha\in A}V_\alpha^h.
\]
Each $V_\alpha^h$ is a finite-dimensional $\widetilde K$-vector space. If
$A$ is replaced by another set $A'$ and $\alpha'\in A'$ represents the class
of $\alpha$, choose $c_\alpha\in K^\times$ with
$\alpha'=\alpha+\ln|c_\alpha|$. Multiplication by $c_\alpha$ identifies the
corresponding summands; another choice changes this identification by a
nonzero residue scalar. The same multiplication rescales both sides of the
inequalities below, so it identifies the entire filtered graded objects.
Thus their dimensions, their $\theta^\star$-degrees, and the semistability
conditions below are independent of $A$ up to graded isomorphism.
This representative-indexed description is used here for filtered
subspaces. For arrow maps of arbitrary real degree, we will instead retain
all real degrees and the scalar action, as in
\namedref{Definition}{def:realgradedmodules}; no identifications of distinct
radii will be implicit in those maps.

Our filtration conventions are as follows. An increasing real filtration is
right-continuous and a decreasing real filtration is left-continuous:
\[
 \mathcal F_{\leq\lambda}
 =\bigcap_{\mu>\lambda}\mathcal F_{\leq\mu},\qquad
 \mathcal F_{<\lambda}
 =\bigcup_{\mu<\lambda}\mathcal F_{\leq\mu},
\]
\[
 \mathcal F_{\geq\lambda}
 =\bigcap_{\mu<\lambda}\mathcal F_{\geq\mu},\qquad
 \mathcal F_{>\lambda}
 =\bigcup_{\mu>\lambda}\mathcal F_{\geq\mu}.
\]
Compatibility with an $\R/\Gamma_K$-grading means that every filtration step
is the direct sum of its intersections with the homogeneous summands. We do
not impose exhaustion or separation: the singular filtrations can have a
persistent kernel or cokernel.

For $T:(V,h_V)\to(W,h_W)$ with split norms, define
\begin{align*}
(\mathcal F^-(T)_{\leq\beta})_\alpha
 &=\{\bar v\in V_\alpha^{h_V}\mid\exists v\text{ representing }\bar v,
                       \ h_W(Tv)\leq e^{\alpha+\beta}\},\\
(\mathcal F^+(T)_{\geq\beta})_\alpha
 &=\{\bar w\in W_\alpha^{h_W}\mid\exists v,\ Tv\text{ represents }\bar w,
                       \ h_V(v)\leq e^{\alpha-\beta}\}.
\end{align*}
The two singular filtrations have the same finite jumps, with the same
graded multiplicities. Listed in decreasing order, these are
$\sigma_1\geq\cdots\geq\sigma_r$, where $r=\rk T$. We call them the
\emph{logarithmic singular values} of $T$; below, \emph{singular values}
means these logarithmic values. The weighted monomial form in
\namedref{Lemma}{lem:monomialcertificate} proves this assertion,
including the stated multiplicities. The persistent kernel of
$\mathcal F^-(T)$ contributes no
finite jump, and the quotient of $\gr_{h_W}W$ by
$\bigcup_{\beta\in\R}\mathcal F^+(T)_{\geq\beta}$ contributes no finite
jump. If $T=0$, including a zero source or target, then
\[
 \mathcal F^-(T)_{\leq\beta}=\gr_{h_V}V,\qquad
 \mathcal F^+(T)_{\geq\beta}=\{0\}
\]
for every $\beta$, and its finite-jump and singular-value lists are empty.

The category of Banach spaces with non-expansive morphisms was
studied by the author and Kremnizer
\cite[Appendix~A.4, especially Definition~A.36 and
pp.~111--112]{BenBassatKremnizer}. Following
\cite[Definitions~6.18--6.20]{HKKP}, let $\mathcal C(Q)$ be the quasi-abelian
category of normed representations and vertexwise non-expansive
morphisms, and let
$\mathcal C^{\mathrm{sp}}(Q)\subseteq\mathcal C(Q)$ be its full subcategory of
split normed representations. Let $\mathcal C_\star(Q)$ be the
quasi-abelian category whose objects consist of finite-dimensional
$\R/\Gamma_K$-graded $\widetilde K$-spaces $V_i$, an increasing
$\R$-filtration $\mathcal F^{a^-}$ on $V_{s(a)}$, and a decreasing
$\R$-filtration $\mathcal F^{a^+}$ on $V_{t(a)}$ for each arrow $a$, all
compatible with the grading. Its morphisms are vertexwise graded
$\widetilde K$-linear maps preserving every incident filtration.

A morphism in a quasi-abelian category is \emph{strict} when its canonical
map from coimage to image is an isomorphism
\cite[Definition~4.1]{BenBassatKremnizer}.  The natural exact structure
consists of strict short exact sequences
\cite[Definition~4.4 and Theorem~4.8]{BenBassatKremnizer}, and
$K_0(\mathcal C_\star(Q))$ is the Grothendieck group generated by object
classes with the relations $[B]=[A]+[C]$ for all such strict short exact sequences
$0\to A\to B\to C\to0$. Concretely, a strict subobject in
$\mathcal C_\star(Q)$ is a graded vertexwise subspace whose incident
filtrations are obtained by intersection with the corresponding ambient
filtrations.

\begin{definition}\label{def:gradedflagobject}
For a split normed representation $(E,\rho,h)$, define
\[
 \GF_{Q,K}(E,\rho,h)
 :=\left(
   (\gr_{h_i}E_i)_{i\in Q_0},
   (\mathcal F^-(\rho(a)),\mathcal F^+(\rho(a)))_{a\in Q_1}
 \right)\in\mathcal C_\star(Q).
\]
For a non-expansive intertwiner $g$, write
\[
 \begin{aligned}
  g=(g_i)&:(E,\rho,h)\longrightarrow(E',\rho',h'),\\
  \GF_{Q,K}(g)&=(\gr g_i)_{i\in Q_0},
  \qquad \gr g_i:\gr_{h_i}E_i\longrightarrow\gr_{h_i'}E_i'.
 \end{aligned}
\]
By \cite[Lemma~6.15]{HKKP}, the maps $\gr g_i$ preserve every incident filtration. This defines a functor
\[
 \GF_{Q,K}:\mathcal C^{\mathrm{sp}}(Q)\longrightarrow\mathcal C_\star(Q),
\]
and $\GF_{Q,K}(E,\rho,h)$ is the \emph{graded-flag object} of $(E,\rho,h)$.
\end{definition}

\begin{remark}\label{rem:coreinterpretation}
Under HKKP's standing spherical-completeness hypothesis every
finite-dimensional norm is split \cite[Proposition~6.9]{HKKP}, and the functor
$\GF_{Q,K}$ is the one they denote by $\operatorname{Core}$.  There are no residual
arrow maps between the graded vertex spaces: each $\rho(a)$ is recorded by
its two singular filtrations.  The arrow maps need not be non-expansive;
non-expansiveness is required only of morphisms in $\mathcal C(Q)$.

For an arbitrary, possibly nonsplit norm, the same graded quotients can be
formed, but their total residual dimension is at most the original dimension,
with equality exactly for split norms.  Indeed, linearly independent
homogeneous initial forms lift to an orthogonal family; their number is at
most the original dimension, with equality exactly when the lifts form an
orthogonal basis.  Thus nonsplit norms can lose dimension under graded
reduction.  We reserve \emph{HKKP harmonic} for split
norms and use \emph{variationally harmonic} for minimizers in $\Nall(E)$. 
Split attainment is what connects these notions.
\end{remark}

For an increasing filtration $\mathcal F$ on a graded space $V$, put
$\gr_\lambda^{\mathcal F}V=\mathcal F_{\leq\lambda}/\mathcal F_{<\lambda}$;
for a decreasing filtration, use
$\gr_\lambda^{\mathcal F}V=\mathcal F_{\geq\lambda}/\mathcal F_{>\lambda}$.
The degree homomorphism
$\theta^\star:K_0(\mathcal C_\star(Q))\to\R$ is
\begin{align}\label{eq:coredegree}
 \theta^\star(V)=\sum_i\theta_i\dim V_i
 +\sum_{a,\lambda}f_a'(\lambda)\bigl(&\dim\gr^{\mathcal F^{a^-}}_\lambda V_{s(a)}
 -\dim\gr^{\mathcal F^{a^+}}_\lambda V_{t(a)}\bigr).
\end{align}
Here dimensions are taken over $\widetilde K$, and the sum over $\lambda$
runs over the finitely many jumps of the incident filtrations.
\begin{lemma}\label{lem:coretotaldegree}
For every split norm $h$ of $E$,
\[
 \theta^\star\bigl(\GF_{Q,K}(E,\rho,h)\bigr)=\theta(E).
\]
\end{lemma}
\begin{proof}
For each arrow, the finite jumps of its source and target singular
filtrations agree, including their multiplicities, by
\namedref{Lemma}{lem:monomialcertificate}.  The two arrow terms in
\fulleqref{eq:coredegree} therefore cancel.  Since $h$ is split,
$\dim_{\widetilde K}\gr_{h_i}E_i=\dim_KE_i$ at every vertex, leaving
$\sum_i\theta_i\dim_KE_i=\theta(E)$.
\end{proof}
\begin{sourcequote}{HKKP,
\cite[Definition~6.20, main text; footnote omitted]{HKKP}.}
A norm $\|\cdot\|$ on a representation $(E,\rho)$ of $Q$ over $K$, i.e.\ a lift to $\mathcal C(Q)$, is \textbf{harmonic} if $\operatorname{Core}(E,\rho,\|\cdot\|)\in\mathcal C_\star(Q)$ is $\theta^\star$-semistable.
\end{sourcequote}
Here a \emph{lift} means a preimage under the forgetful functor
$U:\mathcal C(Q)\to\operatorname{Rep}_K(Q)$,
$(E,\rho,h)\mapsto(E,\rho)$. Thus a lift of $(E,\rho)$ amounts to a
choice of norms $h_i$ on its vertex spaces $E_i$.
Consequently, once $K$, $Q$, $\theta$, $(f_a)_{a\in Q_1}$, and
$(E,\rho)$ are fixed, the only variable in this definition is the
norm $h=(h_i)_{i\in Q_0}$.

Thus, in our notation, a split norm $h$ is harmonic precisely when
$\GF_{Q,K}(E,\rho,h)$ is $\theta^\star$-semistable.  Semistability means
degree zero and nonnegative degree for every strict subobject of
$\mathcal C_\star(Q)$.  In particular,
\namedref{Lemma}{lem:coretotaldegree} shows that harmonicity forces
$\theta(E)=0$.

Now fix the representation $(E,\rho)$ and a nonzero determinant vector in
each $\det E_i$.  The determinant vectors affect the potential only by
an additive constant and are not part of the definition of harmonicity.
HKKP characterize harmonicity through minimization of their convex potential
\begin{equation}\label{eq:HKKPpotential}
 \begin{aligned}
 S_E:\Ndiag(E)=\prod_{i\in Q_0}\Ndiag(E_i)&\longrightarrow\R,\\
 h=(h_i)_i&\longmapsto
 \sum_{a\in Q_1}\sum_{j=1}^{r_a}f_a(\sigma_{a,j}(h))
 -\sum_{i\in Q_0}\theta_i\ell_{E_i}(h_i).
 \end{aligned}
\end{equation}
Here $r_a=\rk\rho(a)$, and $\sigma_{a,j}(h)$ is the $j$th logarithmic
singular value of
\[\rho(a):(E_{s(a)},h_{s(a)})\to(E_{t(a)},h_{t(a)}).\]  Thus $h$ is the
variable, while the representation, weights, ramps, and determinant vectors
are fixed. Convexity is with respect to the canonical geodesics in this norm
space.
For every $c>0$, simultaneous rescaling satisfies
$S_E(ch)=S_E(h)-(\ln c)\theta(E)$. When $\theta(E)=0$, $S_E$ therefore
descends to a function on simultaneous-homothety classes.
The equivalence between harmonicity and minimization is
\cite[Lemma~6.28]{HKKP}.  HKKP work under spherical completeness throughout
\cite[Section~6.3]{HKKP}.  \namedref{Corollary}{cor:splitcriterionallfields}
extends their characterization to split norms over every complete $K$
by a direct first-variation computation.  Existence in the
completed norm space is proved in \namedref{Theorem}{thm:universalvariational}, while
split attainment is proved in \namedref{Theorem}{thm:splitattainment}.

\subsection{Weighted matrices}

For a real weight vector $w=(w_1,\ldots,w_d)\in\R^d$, define
$q_w:K^d\to\R_{\geq0}$ by $q_w(x)=\max_j e^{w_j}|x_j|$.  If
$u\in\R^n$, $w\in\R^m$, and $C\in M_{m\times n}(K)$ is viewed as a map
$C:(K^n,q_u)\to(K^m,q_w)$, then
\begin{equation}\label{eq:weightedmatrixnorm}
 \|C\|_{q_u\to q_w}=\max_{i,j}|C_{ij}|e^{w_i-u_j}.
\end{equation}
An invertible $C$ is an isometry precisely when it and its inverse are
non-expansive.  For a non-expansive $C$, the induced graded map retains exactly
the entries for which $|C_{ij}|e^{w_i-u_j}=1$, together with their leading
coefficients.

\begin{lemma}\label{lem:monomialcertificate}
Let $B:(L^n,q_u)\to(L^m,q_w)$ be linear over a non-Archimedean field $L$. There are weighted isometries $C_s,C_t$ of the
source and target, respectively, such that
\[
 B C_s=C_tD,
\]
where $D$ has at most one nonzero entry in each row and each column, and has
exactly $\rk B$ nonzero entries. In the coordinates defined by $C_s,C_t$,
both singular filtrations are coordinate filtrations. Each nonzero entry
$d$ in position $(i,j)$ has singular value $\ln|d|+w_i-u_j$.
\end{lemma}
\begin{proof}
If $B=0$, take $C_s$ and $C_t$ to be the identity maps and $D=0$.
Otherwise, choose a nonzero entry attaining the maximum in~\fulleqref{eq:weightedmatrixnorm}, in position $(i_0,j_0)$. For $j\ne j_0$, replace the source basis vector $e_j$ by
\[
 e_j-\frac{B_{i_0j}}{B_{i_0j_0}}e_{j_0}.
\]
Maximality gives $|B_{i_0j}/B_{i_0j_0}|e^{u_{j_0}}\leq e^{u_j}$, so this change and its inverse are non-expansive. Next, replace the target basis vector $f_{i_0}$ by
\[
 f_{i_0}+\sum_{i\ne i_0}\frac{B_{ij_0}}{B_{i_0j_0}}f_i.
\]
Again, maximality gives
$|B_{ij_0}/B_{i_0j_0}|e^{w_i}\leq e^{w_{i_0}}$ for every $i\ne i_0$, so
the displayed target-basis change and its inverse are non-expansive. Thus both
basis changes are weighted isometries. The remaining source and target
coordinate spans inherit the restricted weighted max norms, and the pivot row and column have
no other nonzero entries. The identical argument therefore applies
inductively to the complementary block, terminating when that block is zero.
There are at most $\min(m,n)$ pivots. On the resulting monomial matrix the
assertions about the singular filtrations follow directly from their
definitions: a source coordinate of singular value $\sigma$ belongs to
$\mathcal F^-(B)_{\leq\beta}$ exactly when $\sigma\leq\beta$, and the
corresponding target coordinate belongs to $\mathcal F^+(B)_{\geq\beta}$
exactly when $\sigma\geq\beta$. Kernel coordinates belong to every step of
the source filtration, and target coordinates outside the image belong to
none of the target steps. Equal singular values are grouped together.
\end{proof}

\section{Uniform coercivity and existence of minimizers}\label{sec:coercivity}

This section proves a linear lower bound for the potential modulo
simultaneous homothety. The bound is uniform over all norms and supplies
the coercivity needed in \namedref{Theorem}{thm:universalvariational}.

\subsection{The potential on the space of all norms}

For a norm $h$ on $V$, give $\bigwedge^jV$ the quotient of the projective ultrametric tensor norm. The resulting norm $h_{\wedge j}:\bigwedge^jV\to\R_{\geq0}$ is
\begin{equation}\label{eq:exteriornorm}
 h_{\wedge j}(\xi)=
 \inf_{\xi=\sum_\ell v_{\ell1}\wedge\cdots\wedge v_{\ell j}}
 \max_\ell\prod_{q=1}^j h(v_{\ell q}).
\end{equation}
For a split norm $q$, the corresponding exterior norm assigns each wedge
basis vector the product of its basis weights. Comparing $h$ with $q$
as above gives $C^{-j}q_{\wedge j}\leq h_{\wedge j}\leq C^jq_{\wedge j}$,
so \fulleqref{eq:exteriornorm} is positive on every nonzero exterior vector.
For fixed norms $h_V,h_W$, the operator norm is the function
\[
 \|\cdot\|_{h_V\to h_W}:\operatorname{Hom}_K(V,W)\longrightarrow\R_{\geq0},
 \qquad T\longmapsto\sup_{v\ne0}\frac{h_W(Tv)}{h_V(v)}.
\]
The operator norm of the zero map is zero; this convention includes the
unique map with source $V=\{0\}$, for which the displayed supremum is empty.
For a fixed rank-$r$ linear map $T:V\to W$, its exterior power is
$\bigwedge^jT:\bigwedge^jV\to\bigwedge^jW$.  If $r>0$, then
$\bigwedge^jT\ne0$ for $1\leq j\leq r$, so the logarithms below are finite.
Define the real-valued functions
\[
 s_j(T;-,-),\sigma_j(T;-,-):\Nall(V)\times\Nall(W)\to\R
\]
by
\begin{equation}\label{eq:extendedsingularvalues}
 s_j(T;h_V,h_W)=\ln\|\textstyle\bigwedge^jT\|,\qquad
 \sigma_j=s_j-s_{j-1},\qquad s_0=0,\quad 1\leq j\leq r.
\end{equation}
The operator norm here uses the exterior norms induced by $h_V$ and $h_W$.
When the source and target norms are supplied by a norm $h$ on a
representation, we abbreviate these functions as $s_j(T;h)$ and
$\sigma_j(T;h)$.
The tuple $(\sigma_1(T;h_V,h_W),\ldots,\sigma_r(T;h_V,h_W))$ is the
\emph{singular profile} of $T$. For an arrow map $T=\rho(a)$, equipped with
the norms at its source and target vertices, this is its
\emph{arrow singular profile}.
When $r=0$, both the $s_j$-list and the singular-value list are empty.

\begin{lemma}\label{lem:allnormprofile}
The functions $s_j$ in~\fulleqref{eq:extendedsingularvalues} are continuous. More precisely, if the source and target norms change by at most $\delta_s$ and $\delta_t$ in $d_\infty$, then
\begin{equation}\label{eq:profilelipschitz}
 |s_j(T;h_s,h_t)-s_j(T;k_s,k_t)|
 \leq j(\delta_s+\delta_t).
\end{equation}
For split norms, $s_j=\sigma_1+\cdots+\sigma_j$ for the HKKP singular values. Consequently the successive differences in~\fulleqref{eq:extendedsingularvalues} are decreasing for arbitrary norms, and~\fulleqref{eq:HKKPpotential} has a unique continuous extension, denoted by the same symbol,
\[
 S_E:\Nall(E)=\prod_{i\in Q_0}\Nall(E_i)\longrightarrow\R.
\]
\end{lemma}
\begin{proof}
If $e^{-\delta}h\leq k\leq e^\delta h$, the infimum in~\fulleqref{eq:exteriornorm} gives
$e^{-j\delta}h_{\wedge j}\leq k_{\wedge j}\leq e^{j\delta}h_{\wedge j}$.
Applying this at the source and target proves~\fulleqref{eq:profilelipschitz}.
If $(e_i)$ is an orthogonal basis with weights $h(e_i)$, then the wedges
$e_{i_1}\wedge\cdots\wedge e_{i_j}$ form an orthogonal basis whose weights
are $\prod_qh(e_{i_q})$.  Thus, in split orthogonal source and target bases,
the operator norm of $\bigwedge^jT$ is the maximum of the weighted
$j$-minor norms.  Weighted Gaussian elimination identifies its logarithm
with the sum of the largest $j$ HKKP singular values.  Hence the
successive-difference formula agrees with the HKKP singular values on the
dense split locus.  The displayed formula is continuous, so this agreement
also proves uniqueness of the extension; the ordering of the differences
passes to the limit.  Continuity of determinant norms is
\cite[Lemma~2.3]{BE}.
\end{proof}

\begin{lemma}\label{lem:allnormconvexity}
Let $T:V\to W$ have rank $r$.  Each function $s_j(T;-,-)$ is convex along
canonical $d_2$-geodesics in $\Nall(V)\times\Nall(W)$.  Consequently, if
$c_1\geq\cdots\geq c_r\geq0$, then
\[
 (h_V,h_W)\longmapsto\sum_{j=1}^r c_j\sigma_j(T;h_V,h_W)
\]
is geodesically convex. If $f:\R\to\R$ is increasing and convex, then
$\sum_{j=1}^rf(\sigma_j(T;-,-))$ is geodesically convex.
For each finite-dimensional space $V$
and each fixed $0\ne\omega\in\det V$, the function
$h\mapsto\ln h_{\det}(\omega)$ is affine along canonical geodesics.  In particular, the continuous extension
$S_E:\Nall(E)\to\R$ is geodesically convex.
\end{lemma}
\begin{proof}
First suppose that the endpoints are split and choose a common product
apartment, with source weights $u$ and target weights $w$. In its fixed
bases let $T_{I,J}$ denote a $j\times j$ minor of the matrix of $T$. The
weighted exterior operator norm gives
\begin{equation}\label{eq:partialminorformula}
 s_j(T;q_u,q_w)=
 \max_{\substack{|I|=|J|=j\\ \det T_{I,J}\ne0}}
 \left(\ln|\det T_{I,J}|+\sum_{i\in I}w_i-\sum_{k\in J}u_k\right).
\end{equation}
This is a finite maximum of affine functions, hence is convex along the
joining segment. With
$c_{r+1}=0$,
\[
 \sum_{j=1}^r c_j\sigma_j
 =\sum_{j=1}^r(c_j-c_{j+1})s_j,
\]
so the asserted linear combinations are convex.

Let $x_j,z_j$ be the decreasing singular-value lists at the endpoints, let
$y_j$ be the list at parameter $t$, and put
$w_j=(1-t)x_j+tz_j$.  Convexity of the partial sums gives
\[
 \sum_{j=1}^q y_j\leq\sum_{j=1}^q w_j,
 \qquad 1\leq q\leq r.
\]
Thus $y$ is weakly submajorized by $w$.  The majorization inequality for
increasing convex functions \cite[Chapter~3, C.1.b, p.~92]{MOA}, followed by
scalar convexity, gives, as in \cite[Lemma~6.27]{HKKP},
\[
 \sum_{j=1}^r f(y_j)
 \leq\sum_{j=1}^r f(w_j)
 \leq(1-t)\sum_{j=1}^r f(x_j)+t\sum_{j=1}^r f(z_j).
\]
The full determinant logarithms on $V$ and $W$ are affine in the apartment
coordinates.

For arbitrary endpoints $x_0,x_1$, choose split
$x_0^{(n)}\to x_0$ and $x_1^{(n)}\to x_1$.  If $\gamma_n$ and $\gamma$ are
the corresponding constant-speed geodesics, convexity of distance between
geodesics in a CAT(0) space gives
\[
 d_2\bigl(\gamma_n(t),\gamma(t)\bigr)
 \leq(1-t)d_2(x_0^{(n)},x_0)+t\,d_2(x_1^{(n)},x_1).
\]
Hence $\gamma_n(t)\to\gamma(t)$ for every $t$.
\namedref{Lemma}{lem:allnormprofile} permits passage to the limit in the
convexity inequalities; continuity of determinant norms does the same for
the affine identity.  Summing proves convexity of $S_E$.
\end{proof}

\subsection{Closed incidence images and invariant coordinate cuts}

We use Moret-Bailly's proper-image consequence of Greenberg approximation.

\begin{theorem}\label{thm:MBclosedimage}
Let $K$ be a complete, nontrivially valued non-Archimedean field, and let
$\pi:X\to Y$ be a proper morphism of finite-type $K$-schemes. Then
$\pi(X(K))\subseteq Y(K)$ is closed for the valuation topology.
\end{theorem}
\begin{proof}
In \cite[Theorem~1.3]{MB}, the valuation ring $R=K^\circ$ must be Henselian
and the fraction field $\widehat K$ of its completion separable over $K$.
Our valuation has rank one, so completeness implies Henselianity
\cite[Theorems~1.3.1 and 4.1.3(5)]{EnglerPrestel}. Since $R$ is a closed, complete
subring of $K$, we also have $\widehat K=K$.
\end{proof}

Fix the vertex spaces of $E$, put $d_i=\dim E_i$, and define
\[
 \mathscr R_E=\prod_{a\in Q_1}
 \operatorname{Hom}_K(E_{s(a)},E_{t(a)}).
\]
For a dimension vector $e$ with $0\leq e_i\leq d_i$, put
$\theta(e)=\sum_i\theta_i e_i$ and form the incidence scheme
\begin{equation}\label{eq:quiverincidencescheme}
 Z_e=\left\{(\tau,U)\in\mathscr R_E\times
 \prod_i\operatorname{Gr}_K(e_i,E_i):
 \tau(a)U_{s(a)}\subseteq U_{t(a)}\text{ for every }a\in Q_1\right\}.
\end{equation}
Here $\operatorname{Gr}_K(e_i,E_i)$ is the Grassmannian of
$e_i$-dimensional subspaces of $E_i$.
Regard the $E_i$ as trivial bundles on the ambient product and write
$\mathcal U_i$ for the universal subbundles. For each arrow $a$, the
condition is the vanishing of the composite
\[
 \mathcal U_{s(a)}\longrightarrow E_{s(a)}
 \xrightarrow{\tau(a)}E_{t(a)}
 \longrightarrow E_{t(a)}/\mathcal U_{t(a)}.
\]
These are closed conditions, so $Z_e$ is closed. Since the Grassmannians
are projective, its projection $\pi_e:Z_e\to\mathscr R_E$ is proper.
The image $\pi_e(Z_e(K))$ consists exactly of arrow families admitting an
$e$-dimensional subrepresentation. It is closed by
\namedref{Theorem}{thm:MBclosedimage}; for the trivial valuation,
$\mathscr R_E(K)$ is discrete and closedness holds directly.

Given vertexwise bases, a \emph{coordinate cut} is a family
$(U_i)_{i\in Q_0}$ in which each $U_i$ is spanned by a subset of the
chosen basis of $E_i$. Its dimension vector is $(\dim U_i)_i$.
An arrow coefficient from a selected source basis vector to an unselected
target basis vector is called \emph{forbidden}. The family $(U_i)_i$
defines a subrepresentation if and only if all forbidden arrow
coefficients are zero.

Invariance under the arrows is then tested by the forbidden coefficients.
This means that $U=(U_i)_i$ is a subrepresentation exactly when every arrow
sends each selected source basis vector into the span of the selected
target basis vectors. In matrix terms, all entries from selected source
vectors to unselected target vectors must be zero.

For example, suppose $U_s=Ke_1$, $U_t=Kf_1$, and
\[
 \rho(a)e_1=\alpha f_1+\beta f_2.
\]
Here $e_1$ is a source basis vector and $f_1,f_2$ are target basis vectors.
This arrow sends $U_s$ into $U_t$ exactly when $\beta=0$. The coefficient
$\beta$ is ``forbidden'' because it sends a selected source vector toward
an unselected target vector.

\begin{lemma}\label{lem:uniformforbiddenblock}
Fix a product lattice norm $h^0$ and a finite set $\mathscr B$ of dimension
vectors $0\leq e_i\leq d_i$ for which $E$ has no subrepresentation of
dimension vector $e$. There is
$\varepsilon\in(0,1]$ such that, in any vertexwise $h^0$-orthonormal bases,
every coordinate cut of dimension vector in $\mathscr B$ has a forbidden
arrow coefficient of absolute value at least $\varepsilon$.
\end{lemma}
\begin{proof}
Give $\mathscr R_E(K)$ the valuation metric defined by
\[
 \|\tau\|_0=\max_{a\in Q_1}
 \|\tau(a)\|_{h^0_{s(a)}\to h^0_{t(a)}}.
\]
Each closed set $\pi_e(Z_e(K))$, $e\in\mathscr B$, omits $\rho$. Their
finite union therefore avoids an open $\varepsilon$-ball about $\rho$,
for some $\varepsilon\in(0,1]$. If $\mathscr B$ is empty, take
$\varepsilon=1$; a zero-dimensional $\mathscr R_E(K)$ is a point.

If every forbidden coefficient for a coordinate cut $U$ of dimension
$e\in\mathscr B$ had absolute value less than $\varepsilon$, delete those
entries to obtain $\rho'$. Then $U$ is invariant for $\rho'$. All
orthonormal coordinate changes and their inverses are isometries, so
\[
 \|\rho-\rho'\|_0
 =\max_a\max_{\text{forbidden }q,p}|c^a_{qp}|<\varepsilon.
\]
Here $c^a_{qp}$ are the original matrix entries and empty maxima are zero.
This contradicts $\rho'\in\pi_e(Z_e(K))$.
\end{proof}

\subsection{The uniform coercivity theorem}

Put $N_E=\sum_i\dim E_i$. Taking infima over simultaneous homotheties in
the product version of \fulleqref{eq:d2dinftycomparison} gives, also on
the completed quotient,
\[
 \bar d_\infty([h],[k])
 \leq\bar d_2([h],[k])
 \leq\sqrt{N_E}\,\bar d_\infty([h],[k]).
\]
Thus the quotient metrics have the same bounded sets.

\begin{theorem}\label{thm:universalcoercivity}
Fix a finite quiver $Q$, a real weight $\theta=(\theta_i)_{i\in Q_0}$,
a weak ramp $f_a$ for each $a\in Q_1$, and a complete non-Archimedean
field $K$, possibly trivially valued. Let $E=((E_i)_i,\rho)$ be a
finite-dimensional $\theta$-stable representation. Fix nonzero determinant
vectors $\omega_i\in\det E_i$ and a product lattice norm $h^0$. Let $S_E$ be
the HKKP potential in
\fulleqref{eq:HKKPpotential}, extended to $\Nall(E)$ by
\namedref{Lemma}{lem:allnormprofile}. Since $\theta(E)=0$, it descends to
$\overline{\Nall}(E)$.

There exist constants $c>0$ and $C_0\in\R$ such that
\begin{equation}\label{eq:uniformlinearcoercivity}
 S_E(h)\geq c\,\bar d_\infty([h],[h^0])-C_0
 \qquad\text{for every }h\in\Nall(E).
\end{equation}
Consequently, every sublevel of $S_E$ on $\overline{\Nall}(E)$ is bounded
for $\bar d_\infty$, equivalently for $\bar d_2$. The same holds for its
intersection with $\overline{\Ndiag}(E)$.
\end{theorem}
\begin{proof}
Put $d_i=\dim E_i$ and $N=N_E$. If $N\leq1$, the
simultaneous-homothety quotient is a point, and $S_E$ is constant, so
\fulleqref{eq:uniformlinearcoercivity} follows immediately. Assume $N>1$.
At least one arrow is nonzero: otherwise every line at every nonzero
vertex would be a proper subrepresentation, so stability would force
$\theta_i>0$ at all nonzero vertices, contrary to $\theta(E)=0$.

Let
\[
 \mathscr D=\left\{e\in\mathbb Z^{Q_0}:
 0\leq e_i\leq d_i\text{ for every }i,\quad
 0<\sum_i e_i<N\right\},\qquad
 \mathscr B=\{e\in\mathscr D:\theta(e)\leq0\}.
\]
Stability excludes all vectors in $\mathscr B$ as subrepresentations.
Obtain $\varepsilon\in(0,1]$ from
\namedref{Lemma}{lem:uniformforbiddenblock}.

For split $h$, choose common vertexwise orthogonal bases with $h^0$.
Since $h^0$ takes nonzero values in $|K^\times|$, scalar rescaling makes
these bases $h^0$-orthonormal. Write the relative logarithmic weights
as $w_{i,p}$, and simultaneously rescale $h$ to make their minimum zero;
this leaves $S_E$ unchanged. Sort the weights globally:
\[
 0=w_1\leq\cdots\leq w_N=w_{\max},\qquad
 \Delta_k=w_{k+1}-w_k\quad(1\leq k<N).
\]
Let $e^{(k)}\in\mathscr D$ count the first $k$ vectors, ordering ties
arbitrarily. Then $w_{\max}=2\bar d_\infty([h],[h^0])$.

For a nonzero arrow with matrix entries $c^a_{qp}$,
\fulleqref{eq:weightedmatrixnorm} gives
\begin{equation}\label{eq:coercivityoperatorformula}
 \sigma_{a,1}(h)
 =\max_{c^a_{qp}\ne0}
 \bigl(\ln|c^a_{qp}|+w_{t(a),q}-w_{s(a),p}\bigr).
\end{equation}
Set
\[
 H=\max_{\rho(a)\ne0}(\sigma_{a,1}(h))_+,
 \qquad B=-\ln\varepsilon\geq0,
\]
where $(x)_+:=\max\{x,0\}$ is the positive part of $x$.
If $\theta(e^{(k)})\leq0$,
\namedref{Lemma}{lem:uniformforbiddenblock} supplies a forbidden
coefficient of size at least $\varepsilon$. Its source weight is at
most $w_k$ and its target weight at least $w_{k+1}$. Thus
\begin{equation}\label{eq:badcutgapbound}
 \Delta_k\leq\sigma_{a,1}(h)+B\leq H+B
 \qquad\text{whenever }\theta(e^{(k)})\leq0.
\end{equation}

Choose constants, with all extrema over integer dimension vectors,
\[
 \begin{aligned}
 \delta&=\min\bigl(\{1\}\cup
       \{\theta(e):e\in\mathscr D,\ \theta(e)>0\}\bigr)>0,\\
 \beta&=\max_{0\leq e_i\leq d_i}|\theta(e)|,
 \qquad A=(\delta+\beta)(N-1).
 \end{aligned}
\]
Here $\delta=1$ if no vector in $\mathscr D$ has positive degree.
Since $\theta(E)=0$,
\[
 -\sum_{i,p}\theta_i w_{i,p}
 =\sum_{k=1}^{N-1}\theta(e^{(k)})\Delta_k.
\]
Positive-degree cuts contribute at least $\delta\Delta_k$ and the at most
$N-1$ remaining cuts at least $-\beta\Delta_k$. By
\fulleqref{eq:badcutgapbound},
\begin{equation}\label{eq:coercivitydeterminantbound}
 -\sum_{i,p}\theta_i w_{i,p}
 \geq\delta w_{\max}-(\delta+\beta)
       \sum_{\theta(e^{(k)})\leq0}\Delta_k
 \geq\delta w_{\max}-A(H+B).
\end{equation}

Put $r_a=\rk\rho(a)$ and $L=\sum_a r_af_a(-\infty)$. Define a function
$g:\R_{\geq0}\to\R_{\geq0}$ by
\[
 g(0)=0,\qquad
 g(x)=\min_{\rho(a)\ne0}
       \bigl(f_a(x)-f_a(-\infty)\bigr)\quad(x>0).
\]
The arrow contribution is at least $L+g(H)$: when $H>0$, some largest
singular value equals $H$ and all other terms have their usual lower
bounds; when $H=0$, use $L$.

For every $M>0$, the condition $f_a'(x)\to\infty$ gives a number $x_0$
such that $f_a(x)\geq f_a(x_0)+M(x-x_0)$ for $x\geq x_0$.
The finite set of nonzero arrows allows a common $x_0$, so
$g(x)/x\to\infty$. Since $g\geq0$, also near zero,
\[
 m_A:=\inf_{x\geq0}\bigl(g(x)-Ax\bigr)>-\infty.
\]

With $C_{\det}=-\sum_i\theta_i\ell_{E_i}(h_i^0)$,
the determinant contribution is $C_{\det}-\sum_{i,p}\theta_i w_{i,p}$.
Together with \fulleqref{eq:coercivitydeterminantbound}, this gives
\[
 S_E(h)\geq\delta w_{\max}+C_{\det}+L-AB+m_A.
\]
Thus \fulleqref{eq:uniformlinearcoercivity} holds on the split locus with
$c=2\delta$ and $C_0=AB-m_A-C_{\det}-L$.

For arbitrary $h$, take split $h^{(n)}\to h$ in $d_2$. Then
\[
 \left|\bar d_\infty([h^{(n)}],[h^0])
       -\bar d_\infty([h],[h^0])\right|
 \leq d_\infty(h^{(n)},h)\longrightarrow0.
\]
Continuity of $S_E$ extends the bound to all norms. The quotient-metric
comparison gives bounded sublevels.
\end{proof}

\begin{remark}\label{rem:uniformsublevels}
For $N_E>0$, \fulleqref{eq:uniformlinearcoercivity} also gives
\[
 S_E(h)\geq\frac{c}{\sqrt{N_E}}\,
 \bar d_2([h],[h^0])-C_0.
\]
Thus every level $C\in\R$ has a single radius $R_C\geq0$ beyond which
$S_E(h)>C$. The finitely many closed incidence images give a common
excluded neighborhood, bounding a forbidden coefficient below in every
coordinate cut of nonpositive degree, uniformly over all orthonormal bases.
\end{remark}

\subsection{Variational harmonicity}

Throughout this subsection, $E=(E_i,\rho_E(a))$ and $F=(F_i,\rho_F(a))$
denote finite-dimensional representations of the same finite quiver $Q$ over
the complete non-Archimedean field $K$. Their arrow maps are
\[
 \rho_E(a):E_{s(a)}\longrightarrow E_{t(a)},\qquad
 \rho_F(a):F_{s(a)}\longrightarrow F_{t(a)},\qquad a\in Q_1.
\]
The vertex weights $\theta_i$ and arrow ramps $f_a$ are fixed and shared by
the potentials $S_E$, $S_F$, and $S_{E\oplus F}$. Choose nonzero determinant
vectors in every $\det E_i$ and $\det F_i$. The direct sum is formed
vertexwise, with spaces $E_i\oplus F_i$ and arrows
$\rho_E(a)\oplus\rho_F(a)$; its determinant vector is the wedge of the two
chosen vectors. Any other choice changes the corresponding potential by an
additive constant. We write $h=(h_i)\in\Nall(E)$ and
$k=(k_i)\in\Nall(F)$ for the norms on $E$ and $F$, respectively.

\begin{lemma}\label{lem:quotientcompletion}
Choose a nonzero vector $\omega_i\in\det E_i$ at each vertex, taking
$\omega_i=1$ when $E_i=\{0\}$, and put
\[
 N_E=\sum_i\dim E_i,\qquad
 \Delta_E(h)=\sum_i\ell_{E_i}(h_i),\qquad
 \mathscr S_E=\{h\in\Nall(E)\mid\Delta_E(h)=0\}.
\]
If $N_E=0$, the norm spaces and their quotients are singletons.  If
$N_E>0$, every simultaneous-homothety orbit meets $\mathscr S_E$ exactly
once, and the resulting bijection
\[
 \mathscr S_E\longrightarrow\overline{\Nall}(E)
\]
is an isometry for $d_2$ and the quotient metric $\bar d_2$.
The slice $\mathscr S_E$ is closed, convex, and complete; its split points
are dense.  Consequently $(\overline{\Nall}(E),\bar d_2)$ is the metric
completion of $(\overline{\Ndiag}(E),\bar d_2)$, and a set of quotient
classes is $\bar d_2$-bounded if and only if its representatives in
$\mathscr S_E$ are $d_2$-bounded.
\end{lemma}
\begin{proof}
By \namedref{Proposition}{prop:productcompletion}, $\Nall(E)$ is the
complete CAT(0) metric completion of $\Ndiag(E)$.
The assertion for $N_E=0$ follows from our zero-space
conventions, so assume $N_E>0$.  Since
\[
 \Delta_E(ch)=\Delta_E(h)+N_E\ln c \qquad\text{for all }c>0,
\]
every orbit meets $\mathscr S_E$ in exactly one point.  The function
$\Delta_E$ is continuous and affine on canonical geodesics: in a common
product apartment it is the sum of the logarithmic coordinates plus a
determinant-vector-dependent constant, and the general case follows by split
density and endpoint continuity.
Consequently $\mathscr S_E$ is closed and convex, and hence complete.

If $h,k\in\mathscr S_E$ are split, their coordinate differences have sum
zero, so for every $c>0$,
\[
 d_2(h,ck)^2=d_2(h,k)^2+N_E(\ln c)^2.
\]
For arbitrary $h,k\in\mathscr S_E$, choose split approximants
$h^{(n)}\to h$ and $k^{(n)}\to k$, and normalize them by
\[
 \widehat h^{(n)}
 =e^{-\Delta_E(h^{(n)})/N_E}h^{(n)},\qquad
 \widehat k^{(n)}
 =e^{-\Delta_E(k^{(n)})/N_E}k^{(n)}.
\]
Continuity of $\Delta_E$ shows that the normalized split norms still
converge to $h,k$; in particular, split points are dense in
$\mathscr S_E$.  For each fixed $c>0$, pass to the limit in the preceding
identity to obtain it for $h,k$.  Taking the infimum over $c$ proves that
$\mathscr S_E\to\overline{\Nall}(E)$ is an isometry.  The density and
completeness assertions then show that this quotient is the completion of
the quotient split locus.  The final boundedness equivalence follows from
the same isometry.
\end{proof}

When $\theta(E)=0$, the potential is constant on simultaneous-homothety
orbits. Hence minimizing $S_E$ on $\Nall(E)$ is equivalent to minimizing
its restriction to $\mathscr S_E$.

\begin{definition}\label{def:variationalharmonic}
A norm $h\in\Nall(E)$ is \emph{variationally harmonic} if it globally
minimizes the continuous extension of $S_E$. Equivalently, for every
canonical $d_2$-geodesic starting at $h$, the right derivative of $S_E$ at
its initial point is nonnegative. Here right derivatives are understood
as extended one-sided limits in $[-\infty,+\infty]$.
\end{definition}

The equivalence follows at once from convexity: if $S_E(y)<S_E(h)$, the
geodesic from $h$ to $y$ has strictly negative right derivative at its
initial point. At a split norm,
\namedref{Definition}{def:variationalharmonic} agrees with HKKP harmonicity,
as will be proved independently of
\namedref{Theorem}{thm:universalvariational} in part~(i) of
\namedref{Corollary}{cor:splitcriterionallfields}.

\begin{lemma}\label{lem:variationaldirectsum}
If $h$ and $k$ minimize $S_E$ and $S_F$, respectively, then the vertexwise
orthogonal max norm $h\oplus k$, defined by
\[
 (h\oplus k)_i(v,w)=\max\{h_i(v),k_i(w)\},
 \qquad v\in E_i,\quad w\in F_i,
\]
minimizes $S_{E\oplus F}$.
\end{lemma}
\begin{proof}
Let $g=(g_i)\in\Nall(E\oplus F)$.  Write
$g|_E=(g_i|_{E_i})_i\in\Nall(E)$ for the restricted norm.  Under the
canonical identification $(E_i\oplus F_i)/E_i\simeq F_i$, write
$\bar g=g/E=(\bar g_i)_i\in\Nall(F)$ for the quotient norm on
$F\simeq(E\oplus F)/E$, where
$\bar g_i(f):=\inf_{e\in E_i}g_i(e,f)$ for $f\in F_i$.
In general $\bar g$ is not the restriction $g|_F$. Write
$A=\rho_E(a)$, $B=\rho_F(a)$, and $C=A\oplus B$ for a fixed arrow.

First suppose that the source and target norms of $g$ are split.
\namedref{Lemma}{lem:adapted} supplies orthogonal bases adapted to the relevant
copies of $E$, without changing the norms. In these bases the matrix of $C$
is block upper triangular, with diagonal blocks $A$ and $B$ on $E$ and
$F\simeq(E\oplus F)/E$, respectively. If
$r_A=\rk A$, $r_B=\rk B$, and $s_0=0$, combining maximizing minors in the
two diagonal blocks gives
\[
 s_{p+q}(C;g)\geq s_p(A;g|_E)+s_q(B;\bar g),
 \qquad 0\leq p\leq r_A,\quad0\leq q\leq r_B.
\]
For each $j$, the sum of the $j$ largest entries in the union of the
singular-value vectors of $A$ and $B$ is the maximum of the right-hand side
over $p+q=j$. Hence that union is weakly submajorized by the singular-value
vector of $C$. These vectors have the same length because
$\rk C=r_A+r_B$. Since $f_a$ is increasing and convex,
\cite[Chapter~3, C.1.b, p.~92]{MOA} yields, with the quotient norm
$\bar g$ on $F\simeq(E\oplus F)/E$,
\[
 \sum_jf_a(\sigma_j(C;g))\geq
 \sum_jf_a(\sigma_j(A;g|_E))
 +\sum_jf_a(\sigma_j(B;\bar g)).
\]

Now approximate $g$ itself by split norms $g^{(n)}$ and set
$\bar g^{(n)}:=g^{(n)}/E\in\Nall(F)$, using
$F\simeq(E\oplus F)/E$. For each
$g^{(n)}$, \namedref{Lemma}{lem:adapted} supplies the adapted bases used above.
Restriction and quotient are non-expansive for $d_\infty$: inequalities
$e^{-\delta}g\leq g'\leq e^\delta g$ remain true after restriction and
after taking infima over quotient cosets. The following estimates give
convergence in $\Nall(E)$ and $\Nall(F)$, respectively, where
$F\simeq(E\oplus F)/E$ in the quotient term:
\[
 d_\infty(g^{(n)}|_E,g|_E)
 \leq d_\infty(g^{(n)},g)\longrightarrow0,
 \qquad
 d_\infty(\bar g^{(n)},\bar g)
 \leq d_\infty(g^{(n)},g)\longrightarrow0.
\]
By \namedref{Lemma}{lem:allnormprofile}, the arrow inequality therefore passes
to arbitrary $g$. Determinant norms are multiplicative in the restriction--quotient
sequence \cite[Corollary~2.10]{BE}; consequently, with $\bar g$ the quotient
norm on $F\simeq(E\oplus F)/E$,
\[
 S_{E\oplus F}(g)\geq S_E(g|_E)+S_F(\bar g).
\]

Finally choose split $h^{(n)}\to h$ and $k^{(n)}\to k$. For
$h^{(n)}\oplus k^{(n)}$, every arrow is block diagonal between orthogonal
sums, so its singular-value multiset is the sorted union of the two summand
multisets and
\[
 S_{E\oplus F}(h^{(n)}\oplus k^{(n)})
 =S_E(h^{(n)})+S_F(k^{(n)}).
\]
The orthogonal sums converge to $h\oplus k$; continuity therefore gives the
same equality for the possibly nonsplit norms $h,k$. Thus, again viewing
$\bar g$ as the quotient norm on $F\simeq(E\oplus F)/E$,
\[
 S_{E\oplus F}(g)
 \geq S_E(g|_E)+S_F(\bar g)
 \geq S_E(h)+S_F(k)
 =S_{E\oplus F}(h\oplus k),
\]
which proves the claim.
\end{proof}

\begin{theorem}\label{thm:universalvariational}
Fix a finite quiver $Q$, a complete non-Archimedean field $K$,
$\theta\in\R^{Q_0}$, weak ramps $f=(f_a)_{a\in Q_1}$, and a
finite-dimensional representation $E$ of $Q$ over $K$.  For fixed nonzero
determinant vectors in $\det E_i$, let $S_E$ be the resulting potential. Then
\[
\begin{aligned}
 E\text{ is $\theta$-polystable}
 &\Longrightarrow E\text{ admits a variationally harmonic norm for }(\theta,f)\\
 &\Longrightarrow E\text{ is $\theta$-semistable}.
\end{aligned}
\]
If $E$ is $\theta$-polystable, its set of minimizers in the normalized space
$\mathscr S_E$ is a nonempty closed convex subset. For an arbitrary
representation with $\theta(E)=0$, this set is closed and convex but may be empty.
\end{theorem}
\begin{proof}
For the zero representation, the unique norm minimizes the constant
potential, so the assertions are immediate. Assume henceforth that $E\ne0$.
Work in the complete convex slice $\mathscr S_E$ of
\namedref{Lemma}{lem:quotientcompletion}. If $E$ is $\theta$-stable,
\namedref{Theorem}{thm:universalcoercivity} bounds $S_E$ below and makes
its restriction to $\mathscr S_E$ coercive. That restriction is finite,
continuous, and convex, so Ba\v{c}\'ak's minimizer theorem
\cite[Lemma~2.2.19, p.~42]{Bacak} gives a global minimizer. If $E$ is
$\theta$-polystable, apply this argument to its $\theta$-stable summands and use
\namedref{Lemma}{lem:variationaldirectsum}; a simultaneous rescaling places the
result in $\mathscr S_E$. Continuity and geodesic convexity show, for every
representation $E$ with $\theta(E)=0$, that its set of minimizers is closed and convex whenever it
is nonempty.

Conversely, let $h=(h_i)_i\in\Nall(E)$ be a global minimizer.
Simultaneous rescaling forces $\theta(E)=0$.
Let $U\subseteq E$ be a subrepresentation, choose vertexwise complements
$E_i=U_i\oplus B_i$, and choose bases of the two summands. Let $q_{i,t}$
be the split norm assigning weight $e^{-t}$ to the chosen basis of $U_i$
and weight $1$ to that of $B_i$, for $t\geq0$. In these adapted bases
every arrow has block matrix
\[
 \rho(a)=\begin{pmatrix}A_a&C_a\\0&D_a\end{pmatrix}.
\]
The weighted absolute values of entries in the diagonal blocks are
independent of $t$, while those in $C_a$ are multiplied by $e^{-t}$.
Hence, for every nonzero arrow,
$\sigma_{a,1}(q_t)\leq\sigma_{a,1}(q_0)$. Since the singular values are
decreasing and $f_a$ is increasing, its contribution is bounded by
\[
 \sum_{j=1}^{r_a}f_a(\sigma_{a,j}(q_t))
 \leq r_af_a(\sigma_{a,1}(q_0)),
 \qquad r_a=\rk\rho(a),
\]
uniformly in $t$; zero arrows contribute zero. Moreover, the product of
the basis weights gives
\begin{equation}\label{eq:shrinkingdeterminant}
 \ell_{E_i}(q_{i,t})=\ell_{E_i}(q_{i,0})-t\dim U_i.
\end{equation}
Thus, for a constant $C_0$ independent of $t$,
\[
 S_E(q_t)\leq C_0+t\theta(U).
\]
If $\theta(U)<0$, this contradicts $S_E(q_t)\geq S_E(h)$ for large $t$.
The comparison ray need not start at $h$, so the argument also applies
when $h$ is nonsplit. Therefore $E$ is $\theta$-semistable.
\end{proof}

\begin{remark}\label{rem:minimizernonuniqueness}
There is of course no uniqueness.  For one vertex, no arrows, and $\theta=0$, the
potential is constant, so every normalized norm minimizes it.  More
generally, one may independently rescale the norm on each degree-zero
stable summand, using a single factor across all vertices of that summand;
this preserves the potential $S_E$ and produces flat directions modulo common scaling.
\end{remark}

\section{Split minimizers and semistable reduction}
\label{sec:radiusreduction}

\subsection{The split-minimizer criterion}

The first-variation calculation of \cite[Lemma~6.28]{HKKP} can be made
entirely over $K$ at a split norm.  We give its finite matrix proof.

\begin{lemma}\label{lem:splitfirstvariation}
Let $k\in\Ndiag(E)$, put $C=\GF_{Q,K}(E,\rho,k)$, and let $\mathcal G$
be a finite exhaustive separated increasing filtration of its vertex
spaces, compatible with their gradings, with jumps
$\beta_1<\cdots<\beta_m$. We call $\mathcal G$ the \emph{direction filtration},
distinguishing it from the incident singular filtrations of $C$.
Lift homogeneous bases adapted to $\mathcal G$
to orthogonal bases over $K$ and multiply the norm of a vector of direction
weight $\beta$ by $e^{t\beta}$.  The resulting apartment geodesic $k_t$
satisfies
\begin{equation}\label{eq:splitfirstvariation}
 \left.\frac{d}{dt}\right|_{0+}S_E(k_t)
 =-\sum_{j=1}^m\beta_j\,
       \theta^\star\!\left(\gr^{\mathcal G}_{\beta_j}C\right).
\end{equation}
Direction subobjects carry the incident filtrations by intersection,
and successive quotients carry the induced quotient filtrations.
\end{lemma}
\begin{proof}
The orthogonal lifting is justified in
\namedref{Remark}{rem:coreinterpretation}. For any finite filtered graded
space $(H,\mathcal G)$, put
\[
 D_{\mathcal G}(H)=\sum_\beta\beta\dim\gr^{\mathcal G}_\beta H.
\]
For graded subspaces $H_1\subseteq H_2\subseteq H$, use the intersected
filtration on $H_2$ and the induced quotient filtration
\[
 \mathcal G_{\leq\beta}(H_2/H_1)
 =\bigl((H_2\cap\mathcal G_{\leq\beta}H)+H_1\bigr)/H_1.
\]
The corresponding formula with $<\beta$ defines the strict step.
Thus $D_{\mathcal G}$ is additive for the induced
restriction--quotient filtrations.

Fix a nonzero arrow map $T=\rho(a):V\to W$ and write
$\overline V=\gr_kV$, $\overline W=\gr_kW$. Let $\mathcal G_s$ and
$\mathcal G_t$ denote the direction filtrations on $\overline V$ and
$\overline W$, respectively. List the distinct finite singular values of
$T$ as $\lambda_1>\cdots>\lambda_q$. Let $n_l$ be the
cumulative multiplicity through $\lambda_l$, and set
\[
 A_l=\mathcal F^-(T)_{<\lambda_l},\qquad
 B_l=\mathcal F^+(T)_{\geq\lambda_l}.
\]
Both $\overline V/A_l$ and $B_l$ have dimension $n_l$.
By \fulleqref{eq:partialminorformula}, $s_j(T;k_t)$ is a finite maximum
of affine functions.  Its right derivative is the largest slope of an
active minor, namely one attaining the maximum at $t=0$.

At $j=n_l$, the leading exterior map has rank one.  Indeed, in the
monomial form of \namedref{Lemma}{lem:monomialcertificate}, the unique
largest exterior monomial uses the first $n_l$ entries: this index is a
boundary between distinct singular-value groups.  After scaling the
target exterior norm by $e^{-s_{n_l}(T;k)}$, its graded map has source
factor $\det(\overline V/A_l)^\vee$ and target factor $\det B_l$.
This scaling is a real regrading: for the target exterior norm $p$ and
$s=s_{n_l}(T;k)$, its component at degree $\alpha$ is
\[
 \{e^{-s}p\leq e^\alpha\}/\{e^{-s}p<e^\alpha\}
 =\{p\leq e^{\alpha+s}\}/\{p<e^{\alpha+s}\}.
\]
Thus the target grading shifts by $-s$, while its direction filtration
and Pl\"ucker support are unchanged.  This normalization is valid for
every real $s$ and requires no scalar in $K$.
This description is invariant under weighted isometries; at $n_q=\rk T$
it follows also from the rank-one full exterior map.

The active minors are therefore products of the nonzero Pl\"ucker
coordinates of these factors.  The largest target weight sum is
$D_{\mathcal G_t}(B_l)$; the smallest source weight sum is
$D_{\mathcal G_s}(\overline V/A_l)$.  To justify the first assertion,
consider any subspace $B$ in a coordinate space with direction weights
$\beta_j$.  A basis of $B$ adapted to the intersected filtration has
weights $\gamma_1,\ldots,\gamma_{\dim B}$; each basis vector has
coordinates only in steps of weight at most its $\gamma_j$.  Hence every
nonzero Pl\"ucker coordinate of $B$ has weight at most
$\sum_j\gamma_j=D_{\mathcal G}(B)$.  An echelon choice of basis in the
successive coordinate blocks gives a nonzero coordinate attaining this
weight.  For the source factor, the nonzero Pl\"ucker coordinates of
$\det(\overline V/A_l)^\vee$ are indexed by complements of the nonzero
Pl\"ucker coordinates of $A_l$.  Their smallest source weight is therefore
\[
 D_{\mathcal G_s}(\overline V)-D_{\mathcal G_s}(A_l)
 =D_{\mathcal G_s}(\overline V/A_l),
\]
by additivity for the induced quotient filtration.  Since the leading
exterior map has rank one, each nonzero source coordinate can be paired
with each nonzero target coordinate.  Taking the largest active slope
therefore gives
\begin{equation}\label{eq:partialsumfirstvariation}
 s'_{n_l}(0+)=D_{\mathcal G_t}(B_l)
                  -D_{\mathcal G_s}(\overline V/A_l).
\end{equation}

Put $c_l=f_a'(\lambda_l)$, $c_{q+1}=0$.  Differentiating
$\sigma_j=s_j-s_{j-1}$ and grouping equal singular values gives
\[
 \begin{aligned}
 \left(\sum_jf_a(\sigma_j(t))\right)'_{0+}
 &=\sum_l(c_l-c_{l+1})s'_{n_l}(0+)\\
 &=\sum_\lambda f_a'(\lambda)
 \bigl(D_{\mathcal G_t}(\gr^{\mathcal F^+}_\lambda\overline W)-D_{\mathcal G_s}(\gr^{\mathcal F^-}_\lambda\overline V)\bigr).
 \end{aligned}
\]
The second equality follows by telescoping
\fulleqref{eq:partialsumfirstvariation}.  Only finite singular jumps
contribute; a zero arrow contributes an empty sum.
For clarity, the interchange of the incident and direction filtrations
can be seen on a single graded vertex space $H$.  At a finite jump
$\lambda$ of an incident filtration $\mathcal F$, put
\[
 (F_0,F_1)=
 \begin{cases}
  (\mathcal F_{\leq\lambda}H,\mathcal F_{<\lambda}H),
       &\text{if $\mathcal F$ is increasing},\\
  (\mathcal F_{\geq\lambda}H,\mathcal F_{>\lambda}H),
       &\text{if $\mathcal F$ is decreasing}.
 \end{cases}
\]
On $F_0/F_1$, the direction steps are the induced quotient steps
defined above.  On $\gr^{\mathcal G}_\beta H$, the incident steps
corresponding to $F_j$, $j=0,1$, are
\[
 \bigl((F_j\cap\mathcal G_{\leq\beta}H)
                  +\mathcal G_{<\beta}H\bigr)
                  /\mathcal G_{<\beta}H.
\]
Consequently both orders of taking graded pieces have the common
quotient
\[
 \gr^{\mathcal G}_\beta(\gr^{\mathcal F}_\lambda H)
 \simeq
 \frac{F_0\cap\mathcal G_{\leq\beta}H}
 {(F_1\cap\mathcal G_{\leq\beta}H)
                    +(F_0\cap\mathcal G_{<\beta}H)}
 \simeq
 \gr^{\mathcal F}_\lambda(\gr^{\mathcal G}_\beta H).
\]
Their dimensions therefore agree for the source increasing and target
decreasing singular filtrations.  Adding the determinant derivative
$-\sum_i\theta_iD_{\mathcal G_i}(\gr_{k_i}E_i)$ and using
\fulleqref{eq:coredegree} proves \fulleqref{eq:splitfirstvariation}.
\end{proof}

\begin{corollary}\label{cor:splitcriterionallfields}
Fix a finite quiver $Q$, a complete non-Archimedean field $K$,
$\theta\in\R^{Q_0}$, weak ramps $(f_a)_{a\in Q_1}$, and a
finite-dimensional representation $E$ over $K$.  Choose
$0\ne\omega_i\in\det E_i$ and let $S_E$ be the resulting potential.
\begin{enumerate}[label=\textup{(\roman*)},leftmargin=*,nosep]
\item A split norm is HKKP harmonic if and only if it globally minimizes
$S_E$ on the split locus; equivalently, it minimizes its continuous
extension on all norms.  The split harmonic locus is geodesically convex.
\item If $E$ admits a split harmonic norm, then $E$ is
$\theta$-semistable.
\end{enumerate}
\end{corollary}
\begin{proof}
For part~(i), let $C=\GF_{Q,K}(E,\rho,k)$ be semistable.
Since $\theta^\star(C)=0$, summation by parts in
\fulleqref{eq:splitfirstvariation} gives
\[
 \left.\frac{d}{dt}\right|_{0+}S_E(k_t)
 =\sum_{j<m}(\beta_{j+1}-\beta_j)
        \theta^\star(\mathcal G_{\leq\beta_j}C)\geq0.
\]
Every split competitor lies with $k$ in a common product apartment;
its joining segment has a direction filtration given by the coordinate
slopes.  Convexity and the nonnegative initial derivative prove minimality
on the split locus.  Density and continuity extend it to all norms.

Conversely, minimality on the split locus forces $\theta(E)=0$ by
simultaneous scaling in both directions.  Thus $\theta^\star(C)=0$ by
\namedref{Lemma}{lem:coretotaldegree}.  For a strict subobject $U\subseteq C$,
the direction with weight $-1$ on $U$ and $0$ on $C/U$ has derivative
$\theta^\star(U)$ by the lemma.  Minimality makes it nonnegative, proving
semistability of $C$.  Minimization on all norms implies minimization on
the split locus, so both equivalences follow.
Split endpoints have a split joining segment; convexity makes the
segment between minimizers a segment of minimizers.  Part~(i) is therefore
independent of \namedref{Theorem}{thm:universalvariational}.

For part~(ii), part~(i) makes a split harmonic norm a global minimizer
on all norms, and the converse implication of
\namedref{Theorem}{thm:universalvariational} gives $\theta$-semistability.
\end{proof}

\subsection{Semistable reduction with real radius labels}

We now isolate the two bounded-arrow lemmas needed for split attainment.
To avoid sign ambiguity, the two conventions used in the paper are
\[
 \theta\text{-semistability in \fulleqref{eq:stability}}:
 \theta(U)\geq0,\qquad
 \kappa\text{-semistability below}:\kappa(U)\leq0.
\]

\begin{definition}\label{def:realgradedmodules}
The \emph{real-graded residue field} of $K$ is
\[
 \mathcal R=\bigoplus_{\alpha\in\R}\mathcal R_\alpha,
 \qquad
 \mathcal R_\alpha=
 \{x\in K\mid |x|\leq e^\alpha\}/
 \{x\in K\mid |x|<e^\alpha\}.
\]
Multiplication is induced by multiplication in $K$. Thus
$\mathcal R_0=\widetilde K$, and every nonzero homogeneous element is invertible:
the initial form of $x^{-1}$ is the inverse of the initial form of
$x\ne0$. Here \emph{graded field} refers to the invertibility of nonzero homogeneous
elements, not of arbitrary nonzero elements of the underlying ring.
The component $\mathcal R_\alpha$ is zero if $\alpha\notin\Gamma_K$ and
one-dimensional over $\widetilde K$ otherwise.

For a norm $p$ on a finite-dimensional $K$-vector space $V$, put
\[
 M_p(V)=\gr_p^{\R}V
 :=\bigoplus_{\alpha\in\R}
 \{v\in V\mid p(v)\leq e^\alpha\}/
 \{v\in V\mid p(v)<e^\alpha\}.
\]
Scalar multiplication makes this a graded $\mathcal R$-module. A homogeneous
$\mathcal R$-linear map of degree $\eta$ sends the component of degree $\alpha$
to that of degree $\alpha+\eta$; its space is denoted
$\operatorname{Hom}_{\mathcal R}^\eta$. A graded module is \emph{graded free of finite
rank} if it has a finite homogeneous basis, and $\rk_{\mathcal R}$ denotes the number
of basis vectors. All submodules and module morphisms below are graded,
with morphisms of degree zero unless a degree is specified.
\end{definition}

\begin{lemma}\label{lem:realgradedlinearalgebra}
Graded submodules and quotients of a finite-rank graded free $\mathcal R$-module
are graded free, and every graded submodule has a homogeneous complement.
Graded rank is additive in exact sequences and is zero only for the zero
module. If $p$ is a split norm on a finite-dimensional $K$-vector space
$V$, then $M_p(V)$ has graded rank $\dim_KV$, and every
homogeneous basis of $M_p(V)$ lifts to a $p$-orthogonal basis of $V$.

If $p,q$ are split norms on finite-dimensional $K$-vector spaces $V,W$,
respectively, then, for every
$\eta\in\R$, taking initial forms gives an isomorphism
\begin{equation}\label{eq:gradedhomidentification}
 \bigl(\gr^{\R}_{\|\cdot\|_{p\to q}}
          \operatorname{Hom}_K(V,W)\bigr)_\eta
 \simeq\operatorname{Hom}_{\mathcal R}^\eta(M_p(V),M_q(W)).
\end{equation}
These identifications respect addition and composition in the specified
degrees.
\end{lemma}
\begin{proof}
Homogeneous Gaussian elimination works because every nonzero homogeneous
coefficient is invertible. More explicitly, a nonzero homogeneous vector
in a graded submodule has an invertible coefficient in some homogeneous
basis coordinate. Use that coordinate as a pivot, subtract homogeneous
multiples to clear it, and continue in the remaining coordinates. Induction
on the ambient rank gives a homogeneous basis of the submodule extending
to one of the whole module. It also proves the quotient and rank assertions.

The initial forms of a $p$-orthogonal basis form a homogeneous $\mathcal R$-basis
of $M_p(V)$, by the maximum formula for $p$. Conversely, lift a homogeneous
basis to vectors in $V$ of the corresponding norms. In any nonzero linear
combination of the lifts, let $e^\alpha$ be the largest norm of its terms.
The initial forms of those terms in degree $\alpha$ have nonzero sum by
graded linear independence. The combination therefore has norm exactly
$e^\alpha$. This proves orthogonality and linear independence; there are
$\dim_KV$ lifts, so they form a basis of $V$.

If $V$ or $W$ is zero, both Hom spaces are zero. Otherwise choose
orthogonal bases of $V,W$ with logarithmic weights $a_j,c_k$.
For a matrix $(t_{kj})$, its operator norm is
$\max_{k,j}|t_{kj}|e^{c_k-a_j}$. A matrix of norm at most $e^\eta$
therefore has degree-$\eta$ initial coefficient $(k,j)$ in
$\mathcal R_{\eta+a_j-c_k}$. These are exactly the
coefficients of a degree-$\eta$ graded map. Every such finite matrix lifts
coefficientwise to $K$, and its initial form is zero exactly when its
operator norm is strictly less than $e^\eta$. Matrix addition and
multiplication induce the asserted compatibility with addition and
composition.
\end{proof}

For a split norm, the support of $M_p(V)$ is finite modulo $\Gamma_K$,
but need not be finite in $\R$. If $A$ is a set of representatives as in
\namedref{Section}{sec:prelim}, then
\[
 \rk_{\mathcal R} M_p(V)=\sum_{\alpha\in A}\dim_{\widetilde K}(M_p(V))_\alpha.
\]
Indeed, each homogeneous basis vector contributes one dimension in exactly
one representative degree. The same assertion applies to every graded
submodule. This recovers the earlier dimension convention without using
representatives to define arrow maps.

\begin{definition}\label{def:radiuslabelledreduction}
Let $\mathscr Q$ be a finite quiver, let $X$ be a finite-dimensional
representation of $\mathscr Q$, with arrow maps
\[
 \rho_X(e):X_{s(e)}\longrightarrow X_{t(e)},
 \qquad e\in\mathscr Q_1,
\]
and fix a label function
\[
 b:\mathscr Q_1\longrightarrow\R,
 \qquad e\longmapsto b_e.
\]
A norm $p=(p_v)_{v\in\mathscr Q_0}$ of $X$ is \emph{$b$-bounded} if,
for all $e\in\mathscr Q_1$ and $x\in X_{s(e)}$,
\begin{equation}\label{eq:labelledbound}
 p_{t(e)}(\rho_X(e)x)\leq e^{b_e}p_{s(e)}(x).
\end{equation}
The label $b_e$ is logarithmic; the corresponding operator-norm bound is
$e^{b_e}$.
Let $\mathcal A_b$ be the category whose objects consist of finite-rank
graded free $\mathcal R$-modules $Y_v$ at the vertices and homogeneous maps
$Y_{s(e)}\to Y_{t(e)}$ of degree $b_e$ at the arrows. Morphisms are
degree-zero vertexwise maps intertwining all arrows.
If $p$ is split and $b$-bounded, its \emph{radius-labeled graded reduction}
$\gr_p^bX\in\mathcal A_b$ has vertex modules $M_{p_v}(X_v)$ and arrow maps
\[
 \gr_p^{b_e}(\rho_X(e)):
 (M_{p_{s(e)}}(X_{s(e)}))_\alpha\longrightarrow
 (M_{p_{t(e)}}(X_{t(e)}))_{\alpha+b_e},
 \qquad \alpha\in\R,
\]
induced by $\rho_X(e)$ on the displayed quotients. The bound
\fulleqref{eq:labelledbound} makes these maps well-defined, including on
the strict sublevel quotients. The grading retains every real degree; no
scalar identifications between different radii are chosen.

Fix $\kappa=(\kappa_v)_{v\in\mathscr Q_0}\in\R^{\mathscr Q_0}$ with
\[
 \kappa(X):=\sum_{v\in\mathscr Q_0}\kappa_v\dim X_v=0.
\]
Our sign convention for $X$ is that it is
\emph{$\kappa$-semistable} when
\[
 \kappa(U):=\sum_{v\in\mathscr Q_0}\kappa_v\dim U_v\leq0
\]
for every subrepresentation $U\subseteq X$. For $Y\in\mathcal A_b$, put
\[
 d(Y):=\sum_{v\in\mathscr Q_0}\rk_{\mathcal R}Y_v,
 \qquad \kappa(Y):=\sum_{v\in\mathscr Q_0}\kappa_v\rk_{\mathcal R}Y_v.
\]
An object $Y$ with $\kappa(Y)=0$ is $\kappa$-semistable if
$\kappa(U)\leq0$ for every graded subrepresentation $U\subseteq Y$.
\end{definition}

Kernels and cokernels in $\mathcal A_b$ are taken vertexwise. By
\namedref{Lemma}{lem:realgradedlinearalgebra} they again have finite graded
rank, and the arrow maps descend with their prescribed degrees. Thus
the canonical map from coimage to image is an isomorphism: it is the usual
graded-module isomorphism at each vertex and intertwines the induced arrows.
Consequently $\mathcal A_b$ is abelian. A proper inclusion of subobjects increases total
graded rank, so every object has finite length.

\begin{definition}\label{def:gradedslope}
For a nonzero object $Y\in\mathcal A_b$, its \emph{slope} and
\emph{maximal slope} are
\[
 \mu(Y):=\frac{\kappa(Y)}{d(Y)},\qquad
 \mu_{\max}(Y):=
 \max_{\{0\}\ne U\subseteq Y}\mu(U).
\]
The maximum exists: the vertex graded ranks of $U$ are bounded integers,
so only finitely many slopes occur. We call $Y$ slope-semistable if every
nonzero subobject has slope at most $\mu(Y)$. When $\kappa(Y)=0$, this is
exactly $\kappa$-semistability. The zero object is $\kappa$-semistable; no
slope is assigned to it.
\end{definition}

\begin{lemma}\label{lem:maximalslope}
Every nonzero $Y\in\mathcal A_b$ has a unique largest subobject $A(Y)$
of slope $\mu_{\max}(Y)$. It is slope-semistable, every nonzero quotient
of $A(Y)$ has slope at least $\mu_{\max}(Y)$, and every nonzero subobject
of $Y/A(Y)$ has strictly smaller slope.
\end{lemma}
\begin{proof}
Put $\mu_0=\mu_{\max}(Y)$ and choose a subobject $A$ of slope $\mu_0$
having largest total rank. For a subobject $U\subseteq Y$, put
$\nu(U)=\kappa(U)-\mu_0d(U)$, including $\nu(\{0\})=0$.
Then $\nu(U)\leq0$. If $C$ also has slope $\mu_0$, additivity gives
\[
 \nu(A+C)=\nu(A)+\nu(C)-\nu(A\cap C)
          =-\nu(A\cap C)\geq0.
\]
Hence $A+C$ has slope $\mu_0$. Maximality of $d(A)$ forces $A+C=A$,
so $C\subseteq A$. Thus $A=A(Y)$ is unique and largest, and its
slope-semistability follows from the definition of $\mu_0$.
For a nonzero quotient $I$ of $A$, additivity and the slope bound for its
kernel give $\kappa(I)\geq\mu_0d(I)$.

If a nonzero subobject of $Y/A$ had slope at least $\mu_0$, its inverse
image in $Y$ would have slope at least $\mu_0$ and rank strictly larger
than $d(A)$. A strictly larger slope contradicts maximality of $\mu_0$;
equality contradicts the choice of $A$. This proves the strict quotient
bound.
\end{proof}

The subobject $A(Y)$ is the first Harder--Narasimhan (HN) subobject.
Only the properties just proved, not the rest of the HN filtration, will
be needed below.

\begin{definition}\label{def:gradetranslation}
Let $Y\in\mathcal A_b$. For
$\delta\in\R$, define its \emph{grade translation} $T_\delta Y$ by
\[
 (T_\delta Y)_{v,\alpha}=Y_{v,\alpha-\delta},
 \qquad v\in\mathscr Q_0,\quad \alpha\in\R,
\]
with the same $\mathcal R$-action and the reindexed arrow maps. This is an exact
autoequivalence with inverse $T_{-\delta}$. It preserves vertex graded
ranks, hence $\kappa$, slopes, and the subobjects characterized in
\namedref{Lemma}{lem:maximalslope}. For every split
norm $p$, the identity on vectors induces the canonical identification
\[
 M_{e^\delta p}(V)=T_\delta M_p(V).
\]
Indeed, the sublevel of $e^\delta p$ at $e^\alpha$ is the sublevel of $p$
at $e^{\alpha-\delta}$, and likewise for strict sublevels. These
identifications commute with scalar multiplication and all induced arrow
maps.
\end{definition}

\begin{lemma}
\label{lem:radiuslangton}
Suppose that $X$ admits a split $b$-bounded norm.  Then $X$ is
$\kappa$-semistable if and only if it admits a split $b$-bounded norm $p$
for which $\gr_p^bX$ is $\kappa$-semistable.
\end{lemma}

\begin{proof}
Suppose first that $\gr_p^bX$ is $\kappa$-semistable, and let $U\subseteq X$ be a
subrepresentation.  \namedref{Lemma}{lem:adapted} makes the restricted norm on
$U$ split, and $\gr_{p|_U}^bU$ is a graded subrepresentation of $\gr_p^bX$.
Consequently
\[
 \kappa(U)=\kappa(\gr_{p|_U}^bU)\leq0.
\]

Conversely, suppose that $X$ is $\kappa$-semistable and choose a split
$b$-bounded norm $p$. The argument below is a simultaneous refinement, for
arbitrary real labels, of the modification used in
\cite[Proposition~7.4]{HKKP}.
We first adjust the orthogonal lifts by a shear so that rescaling at the
residual radius produces a nonsplit reversed extension. We then compare
maximal-slope subobjects to prove termination.
If $Y:=\gr_p^bX$ is unstable, let $A_0=A(Y)\subseteq Y$ be its
maximal-slope subobject, put $B_0=Y/A_0$, and write
$\mu_0=\mu(A_0)$. Since $\kappa(Y)=0$ and $Y$ is unstable,
$\mu_0>0$ and $\kappa(A_0)>0$. In particular, both $A_0$ and $B_0$
are nonzero. By \namedref{Lemma}{lem:maximalslope}, $A_0$ is
slope-semistable and every nonzero subobject of $B_0$ has slope less than
$\mu_0$.

Choose homogeneous complements to $(A_0)_v$ in $Y_v$ and lift the resulting
homogeneous $\mathcal R$-bases using
\namedref{Lemma}{lem:realgradedlinearalgebra}. This gives orthogonal
decompositions
\[
 X_v=A_v\oplus B_v
\]
whose graded reductions identify with $A_0$ and $B_0$, respectively.
Write the arrow maps in these decompositions as
\[
 \rho_X(e)=
 \begin{pmatrix}
  a_e&\beta_e\\
  \gamma_e&d_e
 \end{pmatrix}.
\]
Thus
\[
 \begin{gathered}
 a_e:A_{s(e)}\longrightarrow A_{t(e)},\qquad
 \beta_e:B_{s(e)}\longrightarrow A_{t(e)},\\
 \gamma_e:A_{s(e)}\longrightarrow B_{t(e)},\qquad
 d_e:B_{s(e)}\longrightarrow B_{t(e)}.
 \end{gathered}
\]
Each block has operator norm at most $e^{b_e}$.  Since $A_0$ is a graded
subrepresentation, the degree-$b_e$ associated-graded component of every
$\gamma_e$ vanishes. Write $\gamma=(\gamma_e)_{e\in\mathscr Q_1}$.
If $\gamma=0$, then $(A_v)_v$ is a $K$-subrepresentation of $X$ with
$\kappa(A)=\kappa(A_0)>0$, contradicting the semistability of $X$.
Therefore $\gamma\ne0$, and
\[
 0<q:=\max_{e\in\mathscr Q_1}
       e^{-b_e}\|\gamma_e\|_{A_{s(e)}\to B_{t(e)}}<1.
\]

Consider the finite-dimensional split normed spaces over $K$
\[
 \begin{aligned}
 C^0&=\prod_{v\in\mathscr Q_0}\operatorname{Hom}_K(A_v,B_v),&
 \|c\|_0&=\max_v\|c_v\|_{A_v\to B_v},\\
 C^1&=\prod_{e\in\mathscr Q_1}
          \operatorname{Hom}_K(A_{s(e)},B_{t(e)}),&
 \|u\|_{1,b}&=\max_e e^{-b_e}
          \|u_e\|_{A_{s(e)}\to B_{t(e)}}.
 \end{aligned}
\]
Thus $\gamma\in C^1$.
Define the coboundary map
\[
 \begin{aligned}
 D:C^0&\longrightarrow C^1,\\
 c&\longmapsto
 \bigl(d_ec_{s(e)}-c_{t(e)}a_e\bigr)_{e\in\mathscr Q_1}.
 \end{aligned}
\]
Give $C^0$ its operator-norm filtration and the $e$-summand of $C^1$ the
filtration defined by $e^{-b_e}\|\cdot\|$, retaining all real degrees.
In the lifted homogeneous bases,
a matrix unit from an $A$-basis vector of grade $\alpha$ to a $B$-basis
vector of grade $\beta$ has degree $\beta-\alpha$ in $C^0$, and degree
$\beta-\alpha-b_e$ in the $e$-summand of $C^1$. Consequently,
by \fulleqref{eq:gradedhomidentification}, for every $\eta\in\R$,
\[
 \begin{aligned}
 (\gr^{\R} C^0)_\eta
 &\simeq\bigoplus_v\operatorname{Hom}_{\mathcal R}^\eta((A_0)_v,(B_0)_v),\\
 (\gr^{\R} C^1)_\eta
 &\simeq\bigoplus_e\operatorname{Hom}_{\mathcal R}^{\eta+b_e}
                ((A_0)_{s(e)},(B_0)_{t(e)}).
 \end{aligned}
\]
In particular, degree $-\delta$ in the first line is the space of
degree-zero vertexwise maps $A_0\to T_\delta B_0$, while degree
$-\delta$ in the second is the space of degree-$b_e$ arrowwise maps
to $T_\delta B_0$. If $\bar a_e,\bar d_e$ denote the arrow maps of $A_0,B_0$,
then
\[
 (\gr^{\R}_{-\delta}D)(\bar c)_e
 =\bar d_e\bar c_{s(e)}-\bar c_{t(e)}\bar a_e.
\]
Hence
\[
 \ker(\gr^{\R}_{-\delta}D)
 =\operatorname{Hom}_{\mathcal A_b}(A_0,T_\delta B_0).
\]
Grade translation preserves slopes. A nonzero morphism
$A_0\to T_\delta B_0$ would have an image that is both a nonzero quotient
of $A_0$, of slope at least $\mu_0$, and a nonzero subobject of
$T_\delta B_0$, of slope strictly less than $\mu_0$. Thus
\[
 \operatorname{Hom}_{\mathcal A_b}(A_0,T_\delta B_0)=\{0\}.
\]

The block bounds give $\|D(c)\|_{1,b}\leq\|c\|_0$. If this inequality were
strict for some nonzero $c$, put $\delta=-\ln\|c\|_0$. Since $C^0$ is
split, $c$ has a nonzero initial form
$\bar c\in(\gr^{\R} C^0)_{-\delta}$, while strictness would give
$(\gr^{\R}_{-\delta}D)(\bar c)=0$, contradicting the preceding vanishing. Thus
\begin{equation}\label{eq:coboundaryisometry}
 \|D(c)\|_{1,b}=\|c\|_0
 \qquad\text{for all }c\in C^0.
\end{equation}

By \fulleqref{eq:coboundaryisometry}, $D$ is an isometric embedding.  The
chosen orthogonal bases make $C^0$ and $C^1$ split, and
\namedref{Lemma}{lem:adapted} gives an orthogonal complement to
$\operatorname{im}D$ in $C^1$.  Projection along this complement, followed
by $D^{-1}$, gives a $K$-linear map
\[
 P:C^1\longrightarrow C^0,
 \qquad \|P\|\leq1,
 \qquad PD=\operatorname{id}_{C^0}.
\]

Define
\[
 \begin{aligned}
 Q:C^0&\longrightarrow C^1,\\
 c&\longmapsto
 \bigl(c_{t(e)}\beta_ec_{s(e)}\bigr)_{e\in\mathscr Q_1}.
 \end{aligned}
\]
Since $\max_{e\in\mathscr Q_1} e^{-b_e}\|\beta_e\|\leq1$, one has
\[
 \|Q(c)\|_{1,b}\leq\|c\|_0^2
\]
and
\[
 \|Q(c)-Q(c')\|_{1,b}
 \leq
 \max\{\|c\|_0,\|c'\|_0\}\,\|c-c'\|_0.
\]
It follows that the map
\[
 \Phi:C^0\longrightarrow C^0,
 \qquad c\longmapsto-P\gamma+PQ(c),
\]
preserves the complete ball
\[
 \mathbb B_q=\{c\in C^0\mid\|c\|_0\leq q\}
\]
and is $q$-Lipschitz there, with $q<1$.  The Banach fixed-point theorem
therefore gives a unique fixed point $c\in\mathbb B_q$.

For each vertex define the automorphism
\[
 S_v:A_v\oplus B_v\longrightarrow A_v\oplus B_v,
 \qquad
 S_v=
 \begin{pmatrix}
  \operatorname{id}_{A_v}&0\\
  c_v&\operatorname{id}_{B_v}
 \end{pmatrix}.
\]
Since $\|c_v\|<1$, both $S_v$ and $S_v^{-1}$ are isometries.  In the new
orthogonal decompositions,
\[
 S_{t(e)}^{-1}\rho_X(e)S_{s(e)}
 =
 \begin{pmatrix}
  a_e+\beta_ec_{s(e)}&\beta_e\\
  z_e&d_e-c_{t(e)}\beta_e
 \end{pmatrix},
\]
where
\[
 z_e=\gamma_e+D(c)_e-c_{t(e)}\beta_ec_{s(e)}.
\]
Since $PD=\operatorname{id}_{C^0}$, the fixed-point equation gives
\begin{equation}\label{eq:residualslice}
 Pz=P\gamma+PD(c)-PQ(c)=P\gamma+c-PQ(c)=0.
\end{equation}
Moreover
\[
 \|z\|_{1,b}
 \leq\max\{\|\gamma\|_{1,b},\|D(c)\|_{1,b},\|Q(c)\|_{1,b}\}
 \leq\max\{q,q,q^2\}=q.
\]
If $z=0$, the subspaces $S_v(A_v\oplus\{0\})$ form a $K$-subrepresentation of
$X$ of weight $\kappa(A_0)>0$, contradicting semistability. Hence
\[
 0<r:=\|z\|_{1,b}\leq q<1.
\]

Put $\lambda=-\ln r>0$, write
$p_{A_v}=p_v|_{A_v}$ and $p_{B_v}=p_v|_{B_v}$, and, for each
$v\in\mathscr Q_0$, define a new split norm by
\[
 p'_v\bigl(S_v(x,y)\bigr)
 =\max\{p_{A_v}(x),r^{-1}p_{B_v}(y)\},
 \qquad\text{for all }(x,y)\in A_v\oplus B_v.
\]
Put
\[
 a'_e=a_e+\beta_ec_{s(e)},\qquad
 d'_e=d_e-c_{t(e)}\beta_e.
\]
Because $\|c\|_0<1$, the perturbations $\beta_ec_{s(e)}$ and
$c_{t(e)}\beta_e$ have normalized norm strictly less than one. With respect
to $p'$,
\[
 e^{-b_e}\|a'_e\|_{p_{A_{s(e)}}\to p_{A_{t(e)}}}\leq1,\qquad
 e^{-b_e}\|d'_e\|_{p_{B_{s(e)}}\to p_{B_{t(e)}}}\leq1,
\]
\[
 e^{-b_e}\|\beta_e\|_{r^{-1}p_{B_{s(e)}}\to p_{A_{t(e)}}}\leq r<1,
 \qquad
 e^{-b_e}\|z_e\|_{p_{A_{s(e)}}\to r^{-1}p_{B_{t(e)}}}\leq1.
\]
Thus $p'$ is $b$-bounded. Scaling each $B_v$-norm by
$r^{-1}=e^\lambda$ replaces $B_0$ by $T_\lambda B_0$. The $\beta_e$-blocks
have zero degree-$b_e$ reduction, the diagonal reductions remain
$\bar a_e$ and the canonically reindexed $\bar d_e$, and at least one
$\bar z_e$ is nonzero because the
finite weighted maximum defining $r$ is attained. Write
$\bar z=(\bar z_e)_e\in(\gr^{\R} C^1)_{-\lambda}$. By the preceding
Hom identifications, these are exactly the degree-$b_e$ maps from $A_0$
to $T_\lambda B_0$. Hence, relative to
$A_0\oplus T_\lambda B_0$,
\[
 \gr_{p'}^b\rho_X(e)=
 \begin{pmatrix}
  \bar a_e&0\\
  \bar z_e&\bar d_e
 \end{pmatrix},
\]
and $Y':=\gr_{p'}^bX$ fits into
\begin{equation}\label{eq:reversedradiusextension}
 0\longrightarrow T_\lambda B_0
 \longrightarrow Y'
 \longrightarrow A_0
 \longrightarrow0.
\end{equation}

Both $D$ and $P$ preserve their norm filtrations, so they induce
degree-zero maps of the real-graded modules. Functoriality of
associated grading carries $PD=\operatorname{id}$ to the same identity on
the associated graded modules. At degree $-\lambda=\ln r$, write these maps as
\[
 \overline D:(\gr^{\R} C^0)_{-\lambda}\longrightarrow(\gr^{\R} C^1)_{-\lambda},
 \qquad
 \overline P:(\gr^{\R} C^1)_{-\lambda}\longrightarrow(\gr^{\R} C^0)_{-\lambda},
\]
with
\[
 \overline P\,\overline D=\operatorname{id},
 \qquad
 \overline P(\bar z)=0
\]
by $PD=\operatorname{id}$ and \fulleqref{eq:residualslice}. Here $\overline P$ is a
map of the graded Hom spaces; it need not intertwine quiver arrows.
A splitting
of \fulleqref{eq:reversedradiusextension} would be the graph of a homogeneous
vertexwise map $\overline h:A_0\to T_\lambda B_0$. The intertwining
equations for that graph are
\[
 \bar z+\overline D(\overline h)=0.
\]
Applying $\overline P$ gives $\overline h=0$, hence $\bar z=0$,
contradicting the choice of $r$. Therefore the extension is nonsplit.

We now prove termination by an elementary slope comparison. Write
$B'=T_\lambda B_0$. For a nonzero subobject $C\subseteq Y'$, let
$C_B=C\cap B'$ and let $I$ be the image of $C$ in $A_0$. The exact sequence
$0\to C_B\to C\to I\to0$ gives
\[
 \kappa(C)-\mu_0d(C)
 =\bigl(\kappa(C_B)-\mu_0d(C_B)\bigr)
  +\bigl(\kappa(I)-\mu_0d(I)\bigr)\leq0.
\]
The first summand is strictly negative if $C_B\ne\{0\}$; the second is
nonpositive because $I\subseteq A_0$. Consequently
$\mu_{\max}(Y')\leq\mu_0$. If equality holds, apply this calculation to
$C=A(Y')$. It forces $C_B=\{0\}$, so $C$ injects into $A_0$ and
$d(C)\leq d(A_0)$. Equality of these ranks would make the image all of
$A_0$: the quotient would have rank zero. The inverse of $C\to A_0$,
followed by $C\subseteq Y'$, would then split
\fulleqref{eq:reversedradiusextension}, a contradiction. Hence either
\[
 \mu_{\max}(Y')<\mu_{\max}(Y),\qquad\text{or}\qquad
 \mu_{\max}(Y')=\mu_{\max}(Y)\ \text{and}\ d(A(Y'))<d(A(Y)).
\]
Thus the pair $(\mu_{\max}(Y),d(A(Y)))$ strictly decreases in lexicographic
order at every nonterminal modification. Its possible values form a
finite set: the vertex graded-rank vector of $A(Y)$ lies in
\[
 \prod_{v\in\mathscr Q_0}\{0,1,\ldots,\dim X_v\},
\]
and both its slope and total rank depend only on that nonzero vector,
not on the real grades. This remains a finite set for arbitrary real
weights $\kappa_v$. Iteration therefore
terminates at a split $b$-bounded norm whose radius-labeled graded
reduction is $\kappa$-semistable.
\end{proof}

\begin{remark}\label{rem:radiusboundary}
\namedref{Lemma}{lem:radiuslangton} imposes no compatibility condition between
the real labels $b_e$ and the value group of $K$.  The proof rescales a
split norm by $r^{-1}\in\R_{>0}$; it never requires
$r\in|K^\times|$.  The argument also applies to loops and multiple arrows.
\end{remark}

\begin{remark}\label{rem:nonsplitblock}
The shear normalization in \fulleqref{eq:residualslice} is essential:
a nonzero lower-left residue block can be a coboundary and hence represent
a split extension. For example, take $K=\mathbf Q_\ell$ for a prime $\ell$,
with uniformizer $\pi=\ell$, and consider the identity arrow
$X_s=X_t=K^2$ with $b=0$ and
$\kappa=(1,-1)$. Give $X_s$ its standard norm and $X_t$ the norm
\[
 p_t(xe_1+ye_2)=\max\{|x|,|\pi|^2|y|\}.
\]
The representation is $\kappa$-semistable, while its reduction has
maximal-slope subobject $A_0$ equal to the source line $\bar e_2$ with zero target.
Choose the orthogonal lifts
\[
 A_s=K(e_2+\pi e_1),\quad B_s=Ke_1,\qquad
 A_t=0,\quad B_t=K^2.
\]
The lower-left block has norm $|\pi|$. Rescaling all $B$-norms by
$|\pi|^{-1}$ makes its leading term $\pi e_1$ nonzero, but this term is
the coboundary of the vertexwise map sending $e_2+\pi e_1$ to $\pi e_1$.
The reversed extension therefore splits. Explicitly,
$e_2=(e_2+\pi e_1)-\pi e_1$ has new source norm $1$ and new target norm
$|\pi|<1$, so its initial source line still has zero arrow. The maximal
slope and the total rank of its maximal-slope subobject remain $(1,1)$.

In the coordinates above, projection onto the $e_1$ component gives
$P\gamma=\pi$ and $\beta=0$. The fixed point is $c=-\pi$, which replaces
the lift by $Ke_2$. The corrected residual has norm $|\pi|^2$;
rescaling at this radius gives a nonsplit reversed extension and a
semistable reduction. Thus nonvanishing of the lower-left block in the
modification of \cite[Proposition~7.4]{HKKP} does not by itself imply
nonsplitting; the exact normalization
\fulleqref{eq:residualslice} ensures it here.
\end{remark}

\begin{definition}\label{def:labelleddetfunctional}
Choose a nonzero determinant vector in every $\det X_v$.  The corresponding
\emph{$\kappa$-weighted determinant functional} is
\[
 J_\kappa:\prod_{v\in\mathscr Q_0}\Nall(X_v)\longrightarrow\R,
 \qquad
 p\longmapsto\sum_{v\in\mathscr Q_0}\kappa_v\ell_{X_v}(p_v).
\]
\end{definition}

\begin{lemma}
\label{lem:labelleddetminimum}
Let $p$ be a split $b$-bounded norm for which $\gr_p^bX$ is
$\kappa$-semistable. Then
\[
 J_\kappa(q)\geq J_\kappa(p)
\]
for every, possibly nonsplit, $b$-bounded norm $q$.
\end{lemma}

\begin{proof}
If $X$ is zero, the assertion is immediate. First suppose that $q$ is
split. At each vertex choose a basis orthogonal for both $p_v$ and $q_v$
\cite[Proposition~1.14]{BE}, and write
\[
 q_v(e_{v,j})=e^{\lambda_{v,j}}p_v(e_{v,j}).
\]
In $p$-orthogonal coordinates, the degree-$b_e$ associated-graded map
retains precisely the entries whose normalized weighted size is one.
Thus, if a matrix coefficient contributes to that residue map, it attains
the arrow bound:
\[
 |(\rho_X(e))_{kj}|p_{t(e)}(e_{t(e),k})
 =e^{b_e}p_{s(e)}(e_{s(e),j}).
\]
Orthogonality and feasibility at $q$ give the corresponding weak
inequality after inserting $e^{\lambda_{t(e),k}}$ and
$e^{\lambda_{s(e),j}}$; cancellation yields
$\lambda_{t(e),k}\leq\lambda_{s(e),j}$. Consequently, for every
$u\in\R$, the graded $\mathcal R$-spans of the initial basis vectors with
$\lambda_{v,j}\leq u$ form a graded subrepresentation $F_{\leq u}$ of
$\gr_p^bX$. Their vertex graded ranks count precisely those selected basis
vectors. Since determinant weights add in the chosen product apartment,
\[
 J_\kappa(q)-J_\kappa(p)
 =\sum_{v,j}\kappa_v\lambda_{v,j}.
\]
If $\eta_1<\cdots<\eta_m$ are the distinct $\lambda_{v,j}$, then
$\kappa(X)=0$ gives the finite-sum identity
\[
 \sum_{v,j}\kappa_v\lambda_{v,j}
 =-\sum_{k=1}^{m-1}(\eta_{k+1}-\eta_k)
       \kappa(F_{\leq\eta_k}).
\]
Equivalently, the integrand below is zero outside the finite range of the
$\lambda_{v,j}$ and
\[
 J_\kappa(q)-J_\kappa(p)
 =-\int_\R\kappa(F_{\leq u})\,du\geq0.
\]

Now let $q$ be arbitrary. Recall that an extension of valued fields is
\emph{immediate} if it induces isomorphisms of value groups and residue
fields, and that a valued field is \emph{maximally complete} if it has no
proper immediate valued-field extension. For valued fields, maximal
completeness is equivalent to spherical completeness
\cite[Section~5, remark following Theorem~1]{Poonen}.
Choose a maximally complete
immediate extension $L/K$. Such an extension exists and is complete
\cite[Corollary~6 and Proposition~6(1)]{Poonen}; it is spherically complete,
and every finite-dimensional norm over $L$ is split
\cite[Proposition~6.9]{HKKP}.

Immediacy identifies the real-graded residue fields
$\mathcal R_K\simeq\mathcal R_L$: for $z\in L^\times$, choose
$a\in K^\times$ with $|a|=|z|$ and $u\in K^\circ$ whose residue equals
that of $z/a$; then $|z-au|<|z|$, so $au$ represents the same homogeneous
class as $z$. For every vertex, inclusion induces a
grade-preserving isomorphism
\[
 M_{p_v}(X_v)\xrightarrow{\sim}
 M_{(p_v)_L}(X_v\otimes_KL).
\]
Indeed, a $p_v$-orthogonal $K$-basis remains orthogonal with the same
weights over $L$ \cite[Proposition~1.25(iv)]{BE}, and its initial forms
give homogeneous bases on both sides over the identified graded residue
fields. The inclusions commute with the arrow maps and their
degree-$b_e$ reductions. Thus $\gr_{p_L}^bX_L$ identifies with
$\gr_p^bX$, and remains $\kappa$-semistable.

Feasibility also passes to scalar extension. For every decomposition
$w=\sum_j a_jx_j$, with $a_j\in L$ and $x_j\in X_{s(e)}$, the
projective scalar-extension norms satisfy
\[
 (q_{t(e)})_L(\rho_X(e)_Lw)
 \leq \max_j |a_j|q_{t(e)}(\rho_X(e)x_j)
 \leq e^{b_e}\max_j |a_j|q_{s(e)}(x_j).
\]
Taking the infimum over all such decompositions proves the required
$b$-bound for $q_L$; the same applies to $p_L$.
Both $p_L$ and $q_L$ are split, so the first part of the proof over $L$
gives
\[
 J_{\kappa,L}(q_L)\geq J_{\kappa,L}(p_L),
\]
where the determinant vectors are transported from $K$.
Determinant norms commute with scalar extension for all norms over a
non-Archimedean field \cite[Corollary~2.10]{BE}. Therefore
\[
 J_{\kappa,L}(q_L)=J_\kappa(q),\qquad
 J_{\kappa,L}(p_L)=J_\kappa(p),
\]
which proves the assertion.
\end{proof}

\section{Supporting hinges and split minimizers}\label{sec:hinges}

The next two lemmas convert the HKKP potential into the determinant
functional of \namedref{Section}{sec:radiusreduction}. They yield split
attainment whenever the extended potential has a global minimizer.
Combined with \namedref{Theorem}{thm:universalvariational}, this gives a
split harmonic norm for every polystable representation over every
complete non-Archimedean field.

\begin{definition}\label{def:maxenvelope}
For finite-dimensional $K$-vector spaces $V,W$, a $K$-linear map
$T:V\to W$, and $c\in\R$, define
\[
 \begin{aligned}
 \operatorname{Env}_{T,c}:\Nall(V)\times\Nall(W)&\longrightarrow\Nall(V),\\
 (p,q)&\longmapsto p\vee e^{-c}T^*q,
 \end{aligned}
\]
where
\[
 (p\vee e^{-c}T^*q)(v)=\max\{p(v),e^{-c}q(Tv)\},
 \qquad v\in V.
\]
\end{definition}

\begin{lemma}\label{lem:hingedeterminant}
Let $V,W$ be finite-dimensional $K$-vector spaces, fix
$0\ne\omega_V\in\det V$, let $T:V\to W$ be a $K$-linear map of rank $r$,
and let $p\in\Nall(V)$ and $q\in\Nall(W)$.
Then, for every $c\in\R$,
\begin{equation}\label{eq:hingedeterminant}
 \ell_V(p\vee e^{-c}T^*q)-\ell_V(p)
 =\sum_{j=1}^r\bigl(\sigma_j(T;p,q)-c\bigr)_+,
\end{equation}
where $(x)_+:=\max\{x,0\}$.
If $p$ and $q$ are split, then $p\vee e^{-c}T^*q$ is split.
\end{lemma}

\begin{proof}
For split $p,q$, weighted Gaussian elimination gives orthogonal source and
target bases in which the nonzero part of $T$ is diagonal.  On the source
basis vectors belonging to the finite singular values, the maximum norm
multiplies the old weight by
$e^{(\sigma_j-c)_+}$ and leaves the kernel weights unchanged.  Taking
determinants proves~\fulleqref{eq:hingedeterminant}.  The same diagonal
description also proves splitness;
alternatively, the maximum norm is the restriction of the split maximum
norm on $V\oplus W$ along the graph embedding
$V\longrightarrow V\oplus W$, $v\longmapsto(v,Tv)$, so
\namedref{Lemma}{lem:adapted} applies.

For arbitrary $p,q$, approximate both uniformly by split norms.  The
maximum norms converge uniformly, determinant norms are continuous, and the
extended singular profiles are continuous by
\namedref{Lemma}{lem:allnormprofile}.  Passing to the limit proves the identity.
\end{proof}

\begin{lemma}\label{lem:supportinghinges}
Fix a finite quiver $Q$, a complete non-Archimedean field $K$,
a finite-dimensional representation $E=(E_i,\rho(a))$ of $Q$ over $K$, a
real stability parameter $\theta=(\theta_i)_{i\in Q_0}$, a weak ramp $f_a$
for each $a\in Q_1$, and nonzero determinant vectors in the spaces
$\det E_i$.  Let $S_E$ be the potential in~\fulleqref{eq:HKKPpotential}
determined by these data, and let $h\in\Nall(E)$ be a global minimizer of
$S_E$. For every arrow $a$, put $r_a=\rk\rho(a)$. If $r_a>0$, let
$d_{a,1}<\cdots<d_{a,m_a}$ be the distinct values among
$\sigma_{a,j}(h)$. There exist positive coefficients $\mu_{a,\nu}$ such
that the convex piecewise-linear function
\begin{equation}\label{eq:hingeexpansion}
 g_a(x)=\sum_{\nu=1}^{m_a}\mu_{a,\nu}(x-d_{a,\nu})_+
\end{equation}
satisfies
\[
 g'_{a,-}(d_{a,\nu})<f_a'(d_{a,\nu})<g'_{a,+}(d_{a,\nu})
 \qquad(1\leq\nu\leq m_a),
\]
where $g'_{a,-}$ and $g'_{a,+}$ denote the left and right derivatives.
Call such functions \emph{strict supporting hinges} for $f_a$ at the
singular profile of $h$.
If $r_a=0$, take $m_a=0$ and $g_a=0$.
For any such choice, the norm $h$ minimizes the function
\[
 \begin{aligned}
 G_E:\Nall(E)&\longrightarrow\R,\\
 k&\longmapsto\sum_{a\in Q_1}\sum_{j=1}^{r_a}
 g_a(\sigma_{a,j}(k))
 -\sum_{i\in Q_0}\theta_i\ell_{E_i}(k_i).
 \end{aligned}
\]
Every global minimizer $k$ of $G_E$
satisfies
\begin{equation}\label{eq:hingeprofileequality}
 \sigma_{a,j}(k)=\sigma_{a,j}(h)
 \qquad\text{for all }a\in Q_1,\ 1\leq j\leq r_a,
\end{equation}
and $S_E(k)=S_E(h)$.
\end{lemma}

\begin{proof}
Put $y_{a,j}=f_a'(\sigma_{a,j}(h))$ and define
\[
 \begin{aligned}
 L_h:\Nall(E)&\longrightarrow\R,\\
 k&\longmapsto\sum_{a,j}y_{a,j}\sigma_{a,j}(k)
 -\sum_i\theta_i\ell_{E_i}(k_i).
 \end{aligned}
\]
For each arrow, $y_{a,1}\geq\cdots\geq y_{a,r_a}>0$.  With
$y_{a,r_a+1}=0$ and
$s_{a,j}=\sigma_{a,1}+\cdots+\sigma_{a,j}$,
\[
 \sum_{j=1}^{r_a}y_{a,j}\sigma_{a,j}
 =\sum_{j=1}^{r_a}(y_{a,j}-y_{a,j+1})s_{a,j}.
\]
By \namedref{Lemma}{lem:allnormconvexity}, the functions $s_{a,j}$ are
geodesically convex on the completed norm space; the vertex determinant
terms are affine, so $L_h$ is geodesically convex.

We claim that $h$ minimizes $L_h$.  Otherwise choose $k$ with
$L_h(k)<L_h(h)$ and let $h_t$ be the canonical geodesic from $h$ to $k$.
Convexity gives
\[
 L_h(h_t)-L_h(h)\leq t\bigl(L_h(k)-L_h(h)\bigr).
\]
The geodesic has $d_2(h_t,h)=O(t)$, hence also
$d_\infty(h_t,h)=O(t)$ by~\fulleqref{eq:d2dinftycomparison}.
\namedeqref{Equation}{eq:profilelipschitz}, applied to successive partial sums,
therefore gives
\[
 |\sigma_{a,j}(h_t)-\sigma_{a,j}(h)|=O(t)
\]
for every $a,j$, while each determinant coordinate is affine in $t$.
Taylor's formula for the $C^1$ ramps, summed over the finite set of singular
values, now gives
\[
 S_E(h_t)-S_E(h)=L_h(h_t)-L_h(h)+o(t),
\]
contradicting the minimality of $h$.

To prove existence of strict supporting hinges, fix an arrow $a$ with
$r_a>0$ and put $u_{a,\nu}=f_a'(d_{a,\nu})$. Since $f_a'$ is positive
and strictly increasing, we may choose real numbers
\[
 0=c_{a,0}<u_{a,1}<c_{a,1}<u_{a,2}<\cdots
 <c_{a,m_a-1}<u_{a,m_a}<c_{a,m_a}
\]
and put $\mu_{a,\nu}=c_{a,\nu}-c_{a,\nu-1}$. The resulting function
in~\fulleqref{eq:hingeexpansion} has
\[
 g'_{a,-}(d_{a,\nu})=c_{a,\nu-1},\qquad
 g'_{a,+}(d_{a,\nu})=c_{a,\nu},
\]
so it satisfies the required strict inequalities.

Now fix any choice of strict supporting hinges. For every arrow with
$r_a>0$, put $u_{a,\nu}=f_a'(d_{a,\nu})$. Convexity and the strict
left and right slope inequalities give
\begin{equation}\label{eq:stricthingesupport}
 g_a(x)\geq g_a(d_{a,\nu})
 +u_{a,\nu}(x-d_{a,\nu}),
\end{equation}
with equality if and only if $x=d_{a,\nu}$.
For an arrow $a$ with $r_a=0$, the singular profile is empty, and we take
$g_a=0$.

Summing~\fulleqref{eq:stricthingesupport} over all singular values gives
\begin{equation}\label{eq:GsupportL}
 G_E(k)-G_E(h)=L_h(k)-L_h(h)+D_h(k),
 \qquad D_h(k)\geq0,
\end{equation}
where $D_h(k)=0$ precisely when~\fulleqref{eq:hingeprofileequality} holds.
Thus $h$ minimizes $G_E$.  Equality at the minimum forces both terms on the
right of~\fulleqref{eq:GsupportL} to vanish.  The arrow singular profiles and
the $\theta$-weighted determinant sum are therefore unchanged, which gives
$S_E(k)=S_E(h)$.
\end{proof}

\begin{theorem}
\label{thm:splitattainment}
Fix a finite quiver $Q$, a complete non-Archimedean field $K$,
a finite-dimensional representation $E=(E_i,\rho(a))$ of $Q$ over $K$,
a real stability parameter $\theta=(\theta_i)_{i\in Q_0}$ satisfying
$\theta(E)=0$, weak ramps $f=(f_a)_{a\in Q_1}$, and nonzero determinant vectors
in $\det E_i$. Let $S_E$ be the resulting potential. If $h\in\Nall(E)$ is a
global minimizer of $S_E$, then
there is $k\in\Ndiag(E)$ such that
\[
 \sigma_{a,j}(k)=\sigma_{a,j}(h)
 \quad\text{for all }a\in Q_1,\ 1\leq j\leq\rk\rho(a),
 \qquad S_E(k)=S_E(h).
\]
Consequently, for these data,
\[
\begin{aligned}
 E\text{ is $\theta$-polystable}
 &\Longrightarrow E\text{ admits a split HKKP harmonic norm for }(\theta,f)\\
 &\Longrightarrow E\text{ is $\theta$-semistable}.
\end{aligned}
\]
Thus \hyperref[quote:hkkpquestion624]{HKKP Question~6.24} has an affirmative
answer over every complete non-Archimedean field.
\end{theorem}

\begin{proof}
Choose strict supporting hinges as in
\namedref{Lemma}{lem:supportinghinges}, and form a new quiver $\widehat Q$.  Its
vertices are the original vertices $i\in Q_0$ together with one auxiliary
vertex $u=(a,\nu)$ for every hinge in~\fulleqref{eq:hingeexpansion}, and
\[
 \widehat E_u=E_{s(a)}.
\]
From $u=(a,\nu)$ draw precisely two arrows
\[
 \widehat E_u\xrightarrow{\ \mathrm{id}\ }E_{s(a)},
 \qquad
 \widehat E_u\xrightarrow{\ \rho(a)\ }E_{t(a)},
\]
with logarithmic radius bounds $0$ and $d_{a,\nu}$, respectively; write $b$ for this
label function on $\widehat Q_1$.  Thus a $b$-bounded norm $p$ satisfies
\[
 p_u\geq p_{s(a)},
 \qquad p_u\geq e^{-d_{a,\nu}}p_{t(a)}\circ\rho(a).
\]
Every arrow of $\widehat Q$ goes from an auxiliary vertex to an original
vertex.

Assign the following determinant weights to the vertices:
\begin{align*}
 \kappa_u&=\mu_{a,\nu},\\
 \kappa_i&=-\theta_i-
 \sum_{\substack{a\in Q_1,\ 1\leq\nu\leq m_a\\s(a)=i}}\mu_{a,\nu}.
\end{align*}
They are balanced:
\[
 \sum_{v\in\widehat Q_0}\kappa_v\dim\widehat E_v
 =-\theta(E)=0.
\]
Transport the determinant vector of $E_{s(a)}$ to each auxiliary copy and
define
\[
 J:\prod_{v\in\widehat Q_0}\Nall(\widehat E_v)\longrightarrow\R,
 \qquad
 p\longmapsto\sum_{v\in\widehat Q_0}
 \kappa_v\ell_{\widehat E_v}(p_v).
\]

For fixed $(a,\nu)$, the least feasible auxiliary norm is given by
\[
 \begin{aligned}
 q_{a,\nu}:\prod_{i\in Q_0}\Nall(E_i)&\longrightarrow
 \Nall(E_{s(a)}),\\
 p^0=(p_i)_i&\longmapsto p_{s(a)}\vee
 e^{-d_{a,\nu}}\rho(a)^*p_{t(a)}.
 \end{aligned}
\]
The determinant hinge identity gives the exact elimination formula
\begin{equation}\label{eq:auxiliaryelimination}
 J\bigl(p^0,(q_{a,\nu}(p^0))_{a,\nu}\bigr)=G_E(p^0).
\end{equation}
For a feasible norm $p$, write $p^0=(p_i)_{i\in Q_0}$ for its
original-vertex part. Every auxiliary norm $p_u$ dominates the
corresponding least one.
Pointwise domination passes to the projective tensor norm and then to its
quotient on the determinant line, so determinant norms are monotone. Since
all $\mu_{a,\nu}$ are positive, this monotonicity and
\namedref{Lemma}{lem:supportinghinges} give
\[
 J(p)\geq G_E(p^0)\geq G_E(h).
\]
Hence
\[
 \widehat h:=\bigl(h,(q_{a,\nu}(h))_{a,\nu}\bigr)
\]
is a global minimizer of $J$ on the full $b$-bounded locus.

We next prove that the fixed representation $\widehat E$ is
$\kappa$-semistable.  Let $U\subseteq\widehat E$ be a subrepresentation.
Choose bases adapted to every $U_v$, put arbitrary split coordinate norms on
the original vertices, and make each auxiliary-source norm sufficiently
large that both of its outgoing bounds hold.  This gives a split feasible
norm adapted to $U$. Relative to the decompositions
$\widehat E_v=U_v\oplus C_v$, every arrow has the block form
\[
 \rho_{\widehat E}(e)=
 \begin{pmatrix}
  a_e&\beta_e'\\
  0&d_e'
 \end{pmatrix}.
\]
Multiply the norm on every $U_v$ by $e^{-t}$ and leave the complementary
coordinate norm fixed. The diagonal block norms stay fixed and the
upper-right block norm is multiplied by $e^{-t}$. More precisely, writing
$p_{U_v}$ and $p_{C_v}$ for the initial summand norms,
\[
 \|\beta_e'\|_{p_{C_{s(e)}}\to e^{-t}p_{U_{t(e)}}}
 =e^{-t}\|\beta_e'\|_{p_{C_{s(e)}}\to p_{U_{t(e)}}}.
\]
The real factor rescales the target norm; every $K$-linear arrow map
stays fixed. Thus all bounds persist. Since the
determinant logarithm at $v$ changes by $-t\dim U_v$, along this ray
\[
 J(p_t)=J(p_0)-t\kappa(U).
\]
The lower bound for $J$ therefore forces $\kappa(U)\leq0$.

The preceding construction, applied with $U=\{0\}$, also supplies an initial
split $b$-bounded norm on $\widehat E$.
\namedref{Lemma}{lem:radiuslangton} now supplies a split $b$-bounded
norm $p$ for which $\gr_p^b\widehat E$ is $\kappa$-semistable.  By
\namedref{Lemma}{lem:labelleddetminimum}, $p$ minimizes $J$ on the full
$b$-bounded locus.  Thus $J(p)\leq J(\widehat h)$, while the preceding global lower
bound gives the reverse inequality; hence
\[
 J(p)=J(\widehat h)=G_E(h).
\]
Let $k=(p_i)_{i\in Q_0}$ be the original-vertex part of $p$; it is split
because $p$ is split. The preceding inequalities give directly
\[
 G_E(h)\leq G_E(k)\leq J(p)=G_E(h).
\]
Thus $G_E(k)=G_E(h)$, and \namedref{Lemma}{lem:supportinghinges} gives equality of all
arrow singular profiles and $S_E(k)=S_E(h)$.

If $E$ is $\theta$-polystable,
\namedref{Theorem}{thm:universalvariational} supplies a minimizer in the
completed norm space.  The
first assertion supplies a split minimizer, and
\namedref{Corollary}{cor:splitcriterionallfields} identifies it with an HKKP
harmonic norm. The second displayed implication is the final assertion of
\namedref{Corollary}{cor:splitcriterionallfields}.
\end{proof}

\section{Examples}\label{sec:examples}

Let $K=\Lambda_{\mathbb C}$ be the non-Archimedean Novikov field consisting
of sums $\sum_{\gamma\in\R}c_\gamma T^\gamma$, with
$c_\gamma\in\mathbb C$, such that only finitely many coefficients below any
given real bound are nonzero.  For a nonzero sum, set
$|\sum c_\gamma T^\gamma|=e^{-\min\{\gamma\mid c_\gamma\ne0\}}$.
The field is complete because, along any Cauchy sequence, the coefficients
below each real bound eventually stabilize.
It is not spherically complete. Indeed, put
$\gamma_n=1-\frac1{n+1}$ and $s_N=\sum_{n=1}^N T^{\gamma_n}$.
The nonempty closed balls
\[
 B_N=\{x\in K\mid |x-s_N|\leq e^{-\gamma_{N+1}}\}
 \qquad\text{for }N\geq1
\]
are nested, since
$|s_{N+1}-s_N|=e^{-\gamma_{N+1}}$ and their radii decrease.
A common point would have coefficient $1$ at every exponent
$\gamma_n<1$, contradicting the defining finiteness condition on
Novikov series. Hence $\bigcap_{N\geq1}B_N=\varnothing$.

\subsection[A two-parameter analytic family of \texorpdfstring{$D_4$}{D4} representations]
{A two-parameter analytic family 
of \texorpdfstring{$D_4$}{D4} representations}
\label{subsec:d4analyticfamily}

Here and below, $(-)^{\mathrm{an}}$ denotes Berkovich analytification.
Put $X=(\mathbf P^1_K\times\mathbf G_{m,K})^{\mathrm{an}}$, let
$p:X\to\mathbf P_K^{1,\mathrm{an}}$ be the first projection, and let
$\lambda$ be the coordinate on the second factor. Write $z$ for the affine
coordinate on $\mathbf A_K^{1,\mathrm{an}}
=\mathbf P_K^{1,\mathrm{an}}\setminus\{\infty\}$.
Let $0$ denote the $K$-rational origin, namely the point with affine
coordinate $z=0$. The space $X$ is a pure
two-dimensional analytic space. Index the vertices by $Q_0=\{0,1,2,3\}$,
with center $0$ and arrows $i\to0$ for $i=1,2,3$.
We index each arrow $i\to0$ by its source $i$.
At $x\in X$, let $K_x=\mathcal H(x)$ be the completed residue field,
with absolute value $|\cdot|_x$. Write $e_1,e_2$ for the standard basis
of $K_x^2$. Take the representation $E_x$ with central space
$E_{0,x}=K_x^2$ and leaf spaces $E_{i,x}=L_{i,x}$, where
\[
 L_{1,x}=K_xe_1,\qquad L_{2,x}=K_xe_2,\qquad
 L_{3,x}\subseteq K_x^2.
\]
Here $L_3=p^*(\mathcal O_{\mathbf P^1_K}(-1)^{\mathrm{an}})$ is the
pullback of the tautological line bundle; its fiber on the affine chart is
$K_x(ze_1+e_2)$, and at infinity it is $K_xe_1$. The arrows are the
inclusions. Fix
\begin{align*}
 &\theta_0=1,\qquad \theta_i=-a_i\quad\text{for }i=1,2,3,\qquad
 (a_1,a_2,a_3)=(1,\tfrac12,\tfrac12),\\
 &f_i(\sigma)=e^\sigma\quad\text{for }i=1,2,3,\\
 &h_{0,x}(ae_1+be_2)=\max\{|a|_x,|\lambda(x)|_x^{-1}|b|_x\},\qquad
 h_{i,x}=a_i^{-1}h_{0,x}|_{L_{i,x}}.
\end{align*}
The vector bundles and inclusion maps form an analytic family. The metrics
are continuous in the following concrete sense: the norm of every local
analytic section is a continuous real-valued function of $x$. For the
central bundle this follows from the displayed maximum of absolute values
of analytic functions; for the leaves it follows by restriction and the
constant factors $a_i^{-1}$. Every fiber metric is split.
At a $K$-rational point the fiber is defined over $K$; at a general point
$x$, it is defined over $K_x$.  Accordingly, all assertions are fiberwise;
the norms do not live on a single fixed $K$-vector space.
This construction places the weighted $D_4$ example of
\cite[Example~6.22]{HKKP} in a two-parameter analytic family.
Let $\mathcal P,\mathcal H,\mathcal S$ denote, respectively, the loci where
$E_x$ is $\theta$-polystable, the displayed norm $h_x$ is HKKP harmonic for
$(\theta,f)$, and $E_x$ is $\theta$-semistable. In particular,
$\mathcal H$ concerns this chosen metric, not the existence of a harmonic metric.

\begin{proposition}\label{prop:d4analyticfamily}
These loci are
\begin{align*}
 \mathcal P&=p^{-1}(\{0\})\cong\mathbf G_{m,K}^{\mathrm{an}},\\
 \mathcal H&=\{x\in p^{-1}(\mathbf A_K^{1,\mathrm{an}})
                  \mid |\lambda z|(x)\leq1\},\\
 \mathcal S&=p^{-1}(\mathbf A_K^{1,\mathrm{an}}).
\end{align*}
The locus $\mathcal P$ is the full fiber over $z=0$, with the second
parameter $\lambda$ unrestricted.  Consequently,
\[
 \mathcal P\subsetneq\mathcal H\subsetneq\mathcal S\subsetneq X.
\]
Here $\mathcal S$ is open in $X$, whereas $\mathcal H$ is a closed analytic
domain in $\mathcal S$.
\end{proposition}

\begin{proof}
For a central line $M\subseteq K_x^2$, the least degree of a
subrepresentation with center $M$ is $1-\sum_{L_{i,x}=M}a_i$.
Full-center proper subrepresentations have positive degree.  For finite
$z\ne0$ the three lines are distinct, hence the fiber is semistable; it is
not polystable because its degree-zero $L_1$-subrepresentation has no
complement containing both $L_2$ and $L_3$.  At $z=0$ it is the direct sum
of the stable degree-zero pieces on $K_xe_1$ and $K_xe_2$.  At infinity,
$L_1=L_3$ has total weight $3/2>1$ and destabilizes.  This proves the
descriptions of $\mathcal P$ and $\mathcal S$.

The three singular values are $\ln a_i$, so their ramp derivatives are
$a_i$.  In~\fulleqref{eq:coredegree} the leaf terms cancel the arrow-source
terms.  Thus a core subobject with graded central part $W$ has degree
$\dim W-\sum_i a_i\dim(W\cap\ell_i)$, where
$\ell_i\subseteq\gr_{h_{0,x}}(K_x^2)$ is the associated-graded image of
$L_{i,x}$ with its restricted central norm, called its \emph{initial line}.
In the orthonormal basis $e_1,\lambda e_2$, the third line is
generated by $\lambda ze_1+\lambda e_2$; its initial line is $\ell_2$, a
line distinct from $\ell_1$ and $\ell_2$, or $\ell_1$, according as
$|\lambda z|$ is less than, equal to, or greater than $1$.
Only the last case is unstable, since then $\ell_1$ has weight $3/2$.
At infinity $\ell_3=\ell_1$, so the same weight also rules out harmonicity.
Thus $h_x$ is harmonic exactly when $|\lambda z|(x)\leq1$.
The analytic change of coordinates $(z,\lambda)\mapsto(\lambda z,\lambda)$
on $\mathcal S$ identifies $\mathcal H$ with the closed unit disc times
$\mathbf G_{m,K}^{\mathrm{an}}$, proving that it is an analytic domain
closed in $\mathcal S$. Fix a finite $K$-rational first parameter $z$. As the second
parameter varies, the restrictions of the central norms to $K^2$ depend
only on the real coordinate
$q=-\ln|\lambda|$ and trace an apartment line in $\Ndiag(K^2)$.
For $z\ne0$, harmonicity is exactly $q\geq\ln|z|$; for $z=0$, every
point of this line is harmonic.
\end{proof}

\begin{corollary}\label{cor:d4harmonicexistence}
In this family, the locus where a split HKKP harmonic norm exists is
\[
 \mathcal H_{\exists}:=
 \{x\in X\mid E_x\text{ admits a split HKKP harmonic norm for }(\theta,f)\}
 =\mathcal S.
\]
At the $K$-rational point $(z,\lambda)=(1,1)$, the representation is
indecomposable and $\theta$-semistable but not $\theta$-polystable, while
the displayed split norm is harmonic and globally minimizes the extended
potential. Thus existence of a global minimizer does not imply
$\theta$-polystability, as in \cite[Example~6.22]{HKKP}.
\end{corollary}
\begin{proof}
Fix $x\in\mathcal S$. On the fiber over $K_x$, choose a real $q$
with $q\geq\ln|z(x)|_x$ if $z(x)\ne0$, and any $q$ if $z(x)=0$.
Replace the central norm by
\[
 g_0(ae_1+be_2)=\max\{|a|_x,e^q|b|_x\},
 \qquad g_i=a_i^{-1}g_0|_{L_{i,x}}.
\]
The core calculation in \namedref{Proposition}{prop:d4analyticfamily}
shows that this split norm is harmonic. Conversely, existence of a split
harmonic norm implies semistability by
\namedref{Corollary}{cor:splitcriterionallfields}(ii).
At $(1,1)$ the three leaf lines are distinct. A nontrivial direct-sum
decomposition would split the center into two lines and place each leaf in
one of them, which is impossible for three distinct lines. The other
assertions at this point follow from
\namedref{Proposition}{prop:d4analyticfamily} and the global-minimizer
criterion \namedref{Corollary}{cor:splitcriterionallfields}(i).
\end{proof}

The potential makes the geometry of these minima explicit. Fix
$x\in\mathcal S$, abbreviate $z=z(x)$, and use determinant vectors
\[
 \omega_0=e_1\wedge e_2,\qquad
 \omega_1=v_1=e_1,\qquad \omega_2=v_2=e_2,\qquad
 \omega_3=v_3=ze_1+e_2.
\]
For an arbitrary norm $g=(g_i)\in\Nall(E_x)$, including a nonsplit
one, put
\[
 u_i=\ln g_0(v_i),\qquad d=\ln(g_0)_{\det}(\omega_0),\qquad
 \sigma_i=u_i-\ln g_i(v_i)\qquad\text{for } i=1,2,3.
\]
The potential is exactly
\begin{equation}\label{eq:d4potential}
 S_{E_x}(g)=\sum_{i=1}^3(e^{\sigma_i}-a_i\sigma_i)+R(g_0),
 \qquad R(g_0)=u_1+\tfrac12u_2+\tfrac12u_3-d.
\end{equation}
Since $v_1\wedge v_2=v_1\wedge v_3=\omega_0$, the defining inequality
for the exterior norm gives $d\leq u_1+u_2$ and $d\leq u_1+u_3$.
Averaging these two inequalities proves $R(g_0)\geq0$ for every central
norm. Each scalar term $e^{\sigma_i}-a_i\sigma_i$ has its unique minimum
at $\sigma_i=\ln a_i$, so
\[
 S_{E_x}(g)\geq\sum_{i=1}^3a_i(1-\ln a_i)=2+\ln2.
\]
This is a lower bound on the space of all norms, not merely on an apartment.
Moreover, varying the norm on a leaf varies only its $\sigma_i$ in
\fulleqref{eq:d4potential}; minimization in the leaf variables therefore forces
$e^{\sigma_i}=a_i=-\theta_i$. This explains the chosen leaf factors.

In the standard central apartment, write
\[
 g_0(ae_1+be_2)=\max\{e^{w_1}|a|_x,e^{w_2}|b|_x\},
 \qquad (w_1,w_2)\in\R^2.
\]
Then $d=w_1+w_2$ and
\[
 R(g_0)=\tfrac12\ln\max\{1,|z|_x e^{w_1-w_2}\}.
\]
For $z\ne0$ this is $\tfrac12(w_1+\ln|z|_x-w_2)_+$; for $z=0$ it is
zero. Taking the minimizing leaf norms $g_i=a_i^{-1}g_0|_{L_{i,x}}$
therefore attains the global minimum precisely when
$w_2-w_1\geq\ln|z|_x$ if $z\ne0$, and for every $(w_1,w_2)$ if $z=0$.
For the displayed family, $w_1=0$ and $w_2=-\ln|\lambda(x)|_x$, giving
again the condition $|\lambda z|(x)\leq1$. In particular, for fixed
$x\in\mathcal S$, the minimizing central norms contain an unbounded half-line modulo
simultaneous homothety, not a single point.

Unlike in \namedref{Subsection}{subsec:nonsplitattainment}, all norms exhibited in this
family are split. The family illustrates the distinction between the
harmonicity of a specified metric and the existence of a harmonic metric,
as well as the geometry of the set of minimizers. It does not exhibit a
nonsplit minimizer or require the
split-attainment construction of \namedref{Section}{sec:hinges}. For this, we have the next example.

\subsection{A nonsplit minimizer and one-step split attainment}
\label{subsec:nonsplitattainment}

Put $V=K^2$, with basis $e_1,e_2$, and set
\[
 \gamma_n=1-\frac1{n+1},\qquad s_N=\sum_{n=1}^N T^{\gamma_n},\qquad
 p(ae_1+be_2)=\lim_{N\to\infty}|a+bs_N|.
\]
For $b\ne0$, the absolute value inside the limit is eventually constant:
the Novikov series $-a/b$ cannot contain all the terms $T^{\gamma_n}$
below exponent $1$, so the first noncancelling term has exponent strictly
less than $1$. Thus $p(ae_1+be_2)>e^{-1}|b|$. For $b=0$ the value is $|a|$.
Homogeneity and the ultrametric inequality pass to the limit, so $p$ is a
norm. Moreover,
\[
 p(e_2-s_Ne_1)=e^{-\gamma_{N+1}}\longrightarrow e^{-1},
 \qquad \inf_{a\in K}p(ae_1+e_2)=e^{-1},
\]
and the infimum is not attained. It is the value of the quotient norm
$\bar p:=p/(Ke_1)\in\Nall(V/Ke_1)$ at the coset $e_2+Ke_1$.
If $p$ were split, an orthogonal basis
adapted to $Ke_1$ would supply a minimum-norm representative of every
class in $V/Ke_1$, by \namedref{Lemma}{lem:adapted}. Hence $p$ is nonsplit.
This also shows that $K$ is not spherically complete.

Take the $A_2$ representation $E_s=E_t=V$ with its single arrow
$\mathrm{id}:E_s\to E_t$, weights $(\theta_s,\theta_t)=(-1,1)$, and
ramp $f(\sigma)=e^\sigma$. Use $e_1\wedge e_2$ as determinant vector at
both vertices, and write $\ell(q)=\ln q_{\det}(e_1\wedge e_2)$ for a norm
$q$ on $V$. For arbitrary source and target norms $h_s,h_t$, the two
singular values satisfy $\sigma_1+\sigma_2=\ell(h_t)-\ell(h_s)$, so
\[
 S_E(h_s,h_t)=\sum_{j=1}^2(e^{\sigma_j}-\sigma_j)\geq2.
\]
For $(h_s,h_t)=(p,p)$ both singular values are zero. Thus $(p,p)$ is a
nonsplit global minimizer; this conclusion follows directly, without invoking
\namedref{Theorem}{thm:splitattainment}.

We now carry out that theorem's construction. At the singular value zero,
choose the strict supporting hinge $g(x)=2x_+$, since
$f'(0)=1\in(0,2)$. Its potential is
\[
 G_E(h_s,h_t)=2\sum_{j=1}^2(\sigma_j)_+
              -\ell(h_t)+\ell(h_s)=|\sigma_1|+|\sigma_2|.
\]
The auxiliary quiver has a new source $u$, spaces $V_u=V_s=V_t=V$, and
identity arrows $V_u\to V_s$ and $V_u\to V_t$, both with logarithmic
radius bound zero.
Its weights and determinant functional are
\[
 (\kappa_u,\kappa_s,\kappa_t)=(2,-1,-1),\qquad
 J(q_u,q_s,q_t)=2\ell(q_u)-\ell(q_s)-\ell(q_t).
\]
Feasibility means $q_u\geq q_s,q_t$. Determinant monotonicity gives
$J\geq0$ on this full feasible locus, and $(p,p,p)$ attains zero.
As in the proof of \namedref{Theorem}{thm:splitattainment}, the reduction
starts instead from a feasible \emph{split} auxiliary norm. Choose
$q(ae_1+be_2)=\max\{|a|,|b|\}$ and $\delta>0$, and start with
\[
 (q_u,q_s,q_t)=(q,e^{-\delta}q,q).
\]
The corresponding original-vertex pair has singular values $(\delta,\delta)$ and is not
harmonic. In the auxiliary reduction the arrow $u\to s$ vanishes, whereas
$u\to t$ is the identity. A nonzero graded subrepresentation with vertex
ranks $d_u,d_s,d_t$ has $d_t\geq d_u$ and therefore
\[
 \frac{2d_u-d_s-d_t}{d_u+d_s+d_t}\leq\frac12,
\]
with equality precisely when $d_s=0$ and $d_t=d_u$. Hence the
maximal-slope subobject $A_0$ equals the full spaces at vertices $u,t$ and
is zero at $s$, with slope $1/2$; its quotient $B_0$ is supported only at $s$
and has slope $-1$.

In the notation of \namedref{Lemma}{lem:radiuslangton}, the lifted pieces
$A_v,B_v$ have disjoint vertex supports, so the vertexwise Hom space
$C^0$ is zero and no shear correction is needed. The residual block is
the identity arrow $u\to s$, of norm $r=e^{-\delta}$. The prescribed
radius modification multiplies the norms on $B$ by $r^{-1}=e^\delta$,
giving $(q,q,q)$. Both reduced arrows are now identities. Every graded
subrepresentation $U$ therefore has
\[
 \rk_{\mathcal R} U_s,\ \rk_{\mathcal R} U_t\geq\rk_{\mathcal R} U_u,
 \qquad \kappa(U)=2\rk_{\mathcal R} U_u-\rk_{\mathcal R} U_s-\rk_{\mathcal R} U_t\leq0.
\]
Thus the reduction is semistable after this single radius modification, and
$J$ has decreased from $2\delta$ to zero. Removing the auxiliary vertex
leaves the split harmonic norm $(q,q)$, by
\namedref{Corollary}{cor:splitcriterionallfields}(i), with the same singular profile
$(0,0)$ and potential value $2$ as the nonsplit minimizer $(p,p)$.
Here the radius modification is nontrivial; only the shear-correction part
of the general construction is vacuous.

\subsection[Nonsplit minima and two-step split attainment for D4]
{Nonsplit minima and two-step split attainment for \texorpdfstring{$D_4$}{D4}}
\label{subsec:d4nonsplit}

The next example couples three arrows through a common central space.
It describes all minimizers, and its split-attainment calculation requires
two radius modifications with different effects. Keep the field $K$ from
above, put $W=K^2$, and consider the $D_4$ representation
\[
 E_0=W\oplus W,\qquad E_1=E_2=E_3=W
\]
with arrows $i\to0$ given by
\begin{equation}\label{eq:d4multiplicityarrows}
 \rho_1(v)=(v,0),\qquad
 \rho_2(v)=(0,v),\qquad
 \rho_3(v)=(v,v).
\end{equation}
Set
\[
 \theta_0=\tfrac32,\qquad \theta_1=\theta_2=\theta_3=-1,
 \qquad f_i(\sigma)=e^\sigma.
\]
Choose a basis $w_1,w_2$ of $W$, put $\omega=w_1\wedge w_2$, and use
$\omega$ at every leaf and
\[
 \Omega=(w_1,0)\wedge(w_2,0)\wedge(0,w_1)\wedge(0,w_2)
\]
at the center. The vector $\Omega$ fixes the central determinant
normalization: for every norm $h_0$ on $E_0$,
\[
 \ell_{E_0}(h_0)=\ln(h_0)_{\det}(\Omega),
\]
and its contribution to $S_E$ is $-\tfrac32\ell_{E_0}(h_0)$.
Together with the leaf determinant vectors $\omega$, this fixes the
additive normalization of the potential used below.
Any two image planes in \fulleqref{eq:d4multiplicityarrows}
are complementary, and
\[
 (\bigwedge^2\rho_i)(\omega)\wedge
 (\bigwedge^2\rho_k)(\omega)=\pm\Omega
 \qquad(1\leq i<k\leq3).
\]
Thus all three decompositions of $E_0$ use the same central determinant
vector, with no change in its absolute size.

The representation is $\theta$-polystable: a basis of $W$ identifies it
with two copies of the $D_4$ representation consisting of three distinct
lines in a two-dimensional center. Each copy is stable for these weights.
Indeed, a central line contains at most one leaf image, giving degree at
least $3/2-1>0$, while a proper full-center subrepresentation omits at
least one leaf and has degree at least $1$.

\begin{proposition}\label{prop:d4minimumlocus}
The global minimum of $S_E$ on $\Nall(E)$ is $6$, and its entire minimum
locus is
\begin{equation}\label{eq:d4minimumlocus}
 H(\nu)=(\nu\oplus\nu,\nu,\nu,\nu),
 \qquad \nu\in\Nall(W),
\end{equation}
where $(\nu\oplus\nu)(v,w)=\max\{\nu(v),\nu(w)\}$.
All six arrow singular values at $H(\nu)$ are zero.
The norm $H(\nu)$ is split exactly when $\nu$ is split.
\end{proposition}
\begin{proof}
For an arbitrary norm $h=(h_0,h_1,h_2,h_3)$, define norms on $W$ and
determinant logarithms by
\[
 \begin{aligned}
 p_i(v)&=h_0(\rho_i(v)),&
 d_i&=\ln(p_i)_{\det}(\omega),\\
 \ell_i&=\ln(h_i)_{\det}(\omega),&
 D&=\ln(h_0)_{\det}(\Omega).
 \end{aligned}
\]
For each pair $i<k$, the corresponding direct-sum norm dominates $h_0$
under the isomorphism $W\oplus W\to E_0$ induced by $\rho_i,\rho_k$.
Determinant monotonicity and the direct-sum determinant identity, using
the preceding wedge identity, therefore give
\[
 D\leq d_i+d_k\qquad(1\leq i<k\leq3).
\]
The determinant identities used here hold for arbitrary non-Archimedean
norms \cite[Corollary~2.10]{BE}. Moreover,
$\sigma_{i,1}+\sigma_{i,2}=d_i-\ell_i$.
This follows in adapted orthogonal coordinates for split norms, and
extends to all norms by density and continuity, including continuity
of restriction. Hence
\begin{equation}\label{eq:d4multiplicitypotential}
 \begin{aligned}
 S_E(h)&=\sum_{i=1}^3\sum_{j=1}^2
             (e^{\sigma_{i,j}}-\sigma_{i,j})+R(h_0),\\
 R(h_0)&=\sum_i d_i-\tfrac32D
        =\tfrac12\sum_{i<k}(d_i+d_k-D)\geq0.
 \end{aligned}
\end{equation}
Since $e^\sigma-\sigma\geq1$, this proves $S_E\geq6$.
At $H(\nu)$, each arrow is an isometric embedding with a non-expansive
left inverse given by one of the central coordinate projections. The
same holds on every exterior power, so its singular values are zero.
Also $p_i=\nu$ and $D=2\ln\nu_{\det}(\omega)$, giving $R=0$ and $S_E=6$.

To identify the equality locus, we use determinant rigidity: if
$r\leq s$ are norms on a finite-dimensional space and their determinant
norms agree, then $r=s$. For any nonzero vector $v$, the quotient norms
satisfy $r/Kv\leq s/Kv$. The restriction--quotient determinant identity
\cite[Corollary~2.10]{BE} therefore shows that $r(v)<s(v)$ would make
the determinant norm of $r$ strictly smaller than that of $s$.

If $S_E(h)=6$, every scalar term in
\fulleqref{eq:d4multiplicitypotential} equals $1$, and every pairwise
determinant gap vanishes. Thus
\[
 \sigma_{i,j}=0,\qquad d_1=d_2=d_3=\tfrac12D.
\]
The norm $p_1\oplus p_2$ dominates $h_0$ and has the same determinant,
so rigidity gives $h_0=p_1\oplus p_2$. Consequently
$p_3=\max\{p_1,p_2\}$ pointwise. Since $p_3$ dominates each of $p_1,p_2$
and all three determinants agree, rigidity gives
$p_1=p_2=p_3=\nu$. Finally,
$\sigma_{i,1}=\ln\|\mathrm{id}:(W,h_i)\to(W,\nu)\|=0$
gives $\nu\leq h_i$, while the full singular sum gives equality of
their determinants. Rigidity yields $h_i=\nu$.

If $\nu$ is split, so is $H(\nu)$. Conversely, if its central norm were
split, its restriction to the first image plane would be split by
\namedref{Lemma}{lem:adapted}, and that restriction is $\nu$.
\end{proof}

In particular, any nonsplit norm on $W$ gives a nonsplit global minimizer
on this representation of dimension vector $(4,2,2,2)$. The norm $p$ of
\namedref{Subsection}{subsec:nonsplitattainment} supplies an explicit
choice, but the description \fulleqref{eq:d4minimumlocus} applies to every
norm on the multiplicity space $W$.

\medskip\noindent\textbf{Two radius modifications.}
We now carry out the construction in the proof of
\namedref{Lemma}{lem:radiuslangton} from a split feasible auxiliary norm.
It takes exactly two radius modifications.
Let $q$ be the standard norm on $W$ in the basis $w_1,w_2$, and write
$X=W\oplus0$ and $Y=0\oplus W$ for the central summands.
At the zero singular values of the minimizers in
\namedref{Proposition}{prop:d4minimumlocus}, choose the strict supporting
hinge $g(\sigma)=2\sigma_+$, since $f_i'(0)=1\in(0,2)$ for every $i$.
The auxiliary quiver has three new vertices $u_i$, each carrying $W$,
and arrows
\[
 W_{u_i}\xrightarrow{\mathrm{id}}E_i,\qquad
 W_{u_i}\xrightarrow{\rho_i}E_0,\qquad i=1,2,3,
\]
where $\rho_i$ are the maps in
\fulleqref{eq:d4multiplicityarrows}. All six logarithmic radius bounds
are zero. The determinant weights and functional are
\[
\begin{aligned}
\kappa_{u_i}&=2,\qquad \kappa_i=-1,\qquad \kappa_0=-\tfrac32,\\
J(p)&=2\sum_{i=1}^3\ell_W(p_{u_i})
       -\sum_{i=1}^3\ell_W(p_i)-\tfrac32\ell_{E_0}(p_0).
\end{aligned}
\]
The determinant vectors on the copies of $W$ are $\omega$ and that on
$E_0$ is $\Omega$, as above. The total weight is zero.

Fix $x,\eta>0$. For $t\in\R$, put
\[
 Q_t(v,w)=\max\{e^tq(v),q(w)\}.
\]
Start with the split feasible norm $p^{(0)}$ given by
\[
\begin{aligned}
p^{(0)}_0&=Q_x,\\
(p^{(0)}_{u_1},p^{(0)}_{u_2},p^{(0)}_{u_3})&=(e^xq,q,e^xq),\\
(p^{(0)}_1,p^{(0)}_2,p^{(0)}_3)&=(e^xq,e^{-\eta}q,e^xq).
\end{aligned}
\]
Its reduction is the direct sum of three components: the identity
$u_2\to Y$, the component on $u_1,u_3,E_1,E_3,X$ in which both central
arrows have image in $X$, and the isolated vertex $E_2$. The arrows
$u_1\to E_1$ and $u_3\to E_3$ are identities; $u_2\to E_2$ vanishes.

The maximal slope of the first component is $1/4$. Indeed, if its
source and central ranks are $a$ and $b$, then $b\geq a$ and
\[
 \kappa-\tfrac14d=\tfrac74(a-b)\leq0.
\]
For a subobject of the second component, let $a_1,a_3$ be its auxiliary
source ranks, $b$ its central rank, and $l_1,l_3$ its original leaf
ranks. The arrows give $a_1,a_3\leq b$ and $l_i\geq a_i$, so
\[
 10\kappa-d
 =19(a_1+a_3)-16b-11(l_1+l_3)
 \leq8(a_1+a_3)-16b\leq0.
\]
Its maximal slope is therefore $1/10$, attained by the whole component.
The isolated leaf has slope $-1$. Intersecting a subobject with one
direct summand and projecting to the remaining summands shows that these
bounds also control every subobject of the direct sum. Consequently the
largest maximal-slope subobject $A_0$ consists of the full spaces at
$u_2$ and $Y$, with zero spaces elsewhere:
\[
 d(A_0)=4,\qquad \kappa(A_0)=1,\qquad \mu(A_0)=\tfrac14.
\]

For this splitting, the vertexwise Hom space in
\namedref{Lemma}{lem:radiuslangton} is
$C^0=\operatorname{Hom}_K(Y,X)$, which is nonzero. Identifying both
summands with $W$, its norm is $e^x\|\cdot\|_{q\to q}$.
The only summands of $C^1$ come from $u_2\to X$ and $u_2\to E_2$,
and
\[
 D(c)=(-c,0),\qquad \gamma=(0,\mathrm{id}_W).
\]
Projection onto the first summand with a minus sign gives a contractive
left inverse $P$ of $D$, and $P\gamma=0$. The block $\beta$ from
$u_3$ to $Y$ is nonzero, but the quadratic term $Q$ is identically zero:
$C^0$ is supported only at the central sink, so every $c_{s(e)}$ is
zero. Thus the fixed point is $c=0$, and the residual radius is
\[
 r_1=e^{-\eta}.
\]
The leading residual is the identity $u_2\to E_2$, in a summand on
which every coboundary vanishes. It therefore gives a nonsplit reversed
extension, as required by \fulleqref{eq:reversedradiusextension}.

The first modification multiplies all complementary norms by $e^\eta$.
Writing $y=x+\eta$, the resulting norm $p^{(1)}$ is
\[
\begin{aligned}
p^{(1)}_0&=Q_y,\\
(p^{(1)}_{u_1},p^{(1)}_{u_2},p^{(1)}_{u_3})&=(e^yq,q,e^yq),\\
(p^{(1)}_1,p^{(1)}_2,p^{(1)}_3)&=(e^yq,q,e^yq).
\end{aligned}
\]
All three auxiliary identity arrows are now nonzero in the reduction,
while the reduced images of $u_1$ and $u_3$ still coincide in $X$.
The reduction splits into the component on $u_1,u_3,E_1,E_3,X$ and
the component on $u_2,E_2,Y$. The former has maximal slope $1/10$ by
the preceding calculation. Every nonzero subobject of the latter has
negative weight: with ranks $a_2,b,l_2$, the inequalities
$b,l_2\geq a_2$ give
\[
 \kappa=2a_2-\tfrac32b-l_2
       =-\tfrac12a_2-\tfrac32(b-a_2)-(l_2-a_2)<0.
\]
Thus the largest maximal-slope subobject $A_1$ is the full first
component, and
\[
 d(A_1)=10,\qquad \kappa(A_1)=1,\qquad \mu(A_1)=\tfrac1{10}.
\]

Now $C^0=\operatorname{Hom}_K(X,Y)$, with norm
$e^{-y}\|\cdot\|_{q\to q}$. The two summands of $C^1$ correspond to
the central arrows from $u_1$ and $u_3$, and
\[
 D(c)=(-c,-c),\qquad \gamma=(0,\mathrm{id}_W).
\]
Again, negative projection onto the first summand is a contractive left
inverse with $P\gamma=0$, and $Q=0$. The fixed point is $c=0$ and
\[
 r_2=e^{-y}=e^{-(x+\eta)}.
\]
The leading residual $(0,\overline{\mathrm{id}}_W)$ does not lie in the
coboundary image $\{(-\bar c,-\bar c)\}$. This reversed extension is
therefore nonsplit as well. Multiplying the complementary norms on
$u_2,E_2,Y$ by $e^y$ gives $p^{(2)}$, for which every norm on a copy of
$W$ is $e^yq$ and the central norm is $e^y(q\oplus q)$.

The final reduction is semistable directly. For any graded
subrepresentation, write $a_i$ for the auxiliary source ranks, $l_i$
for the original leaf ranks, and $b$ for the central rank. The identity
arrows give $l_i\geq a_i$. Each pair of central image planes is
complementary, so $a_i+a_j\leq b$ for $i\ne j$. Summing these three
inequalities gives $\sum_i a_i\leq3b/2$, and hence
\[
 \kappa=2\sum_i a_i-\sum_i l_i-\tfrac32b
 \leq\sum_i a_i-\tfrac32b\leq0.
\]
The resulting original-vertex norm is the split minimizer
$H(e^yq)$ of \namedref{Proposition}{prop:d4minimumlocus}.
It has the same six zero singular values and the same potential value
$6$ as every nonsplit minimizer $H(\nu)$ constructed above.

Evaluating determinant norms at $\omega$ and $\Omega$ multiplies the
respective orthogonal basis weights. Since $q(w_j)=1$, this gives
\[
 \begin{aligned}
 \ell_W(e^tq)&=2t\qquad(t\in\R),\\
 \ell_{E_0}(Q_x)&=2x,\qquad
 \ell_{E_0}(Q_y)=2y,\qquad
 \ell_{E_0}\bigl(e^y(q\oplus q)\bigr)=4y.
 \end{aligned}
\]
Substituting these determinant logarithms into $J$ and evaluating $S_E$
on the original vertices gives the following comparison. The slope and
total rank are those of the maximal-slope subobject in the auxiliary
reduction, and each residual radius
is used in the next modification.
\begin{center}
\begingroup
\small
\setlength{\tabcolsep}{6pt}
\renewcommand{\arraystretch}{1.15}
\begin{tabular}{@{}ccccc@{}}
\toprule
Stage & \shortstack{Maximal slope\\and total rank}
      & \shortstack{Next residual\\radius} & $J$ & $S_E$ \\
\midrule
$p^{(0)}$ & $(\tfrac14,4)$ & $e^{-\eta}$
          & $x+2\eta$ & $4+2e^\eta+x-2\eta$ \\
$p^{(1)}$ & $(\tfrac1{10},10)$ & $e^{-y}$
          & $y$ & $6+y$ \\
$p^{(2)}$ & Semistable & --- & $0$ & $6$ \\
\bottomrule
\end{tabular}
\endgroup
\end{center}
At each stage the auxiliary source norms are the least feasible ones,
so \fulleqref{eq:auxiliaryelimination} identifies $J$ with the hinge
potential on the original vertices. After the first modification all
original arrow singular values are already zero, while the original
potential is still strictly greater than its minimum $6$.
The coincident reduced central images still contribute a determinant defect;
the second modification removes it. The table also shows that the first
modification can increase $S_E$ for small $\eta>0$.
The decreasing functionals in this construction are $J$ and the
supporting-hinge potential. Both modifications act on the shared central
space: the first activates the delayed leaf arrow and enlarges the
central gap, while the second separates the colliding reduced images.

\renewcommand{\refname}{Bibliography}

\end{document}